\documentclass[a4paper,11pt]{article}
\usepackage[latin1]{inputenc}
\usepackage[T1]{fontenc}

\usepackage{array}
\usepackage{fancyhdr}
\usepackage{longtable}
\usepackage{tabularx}
\usepackage{rotating}

\usepackage{amsmath}
\usepackage{amsthm}
\usepackage{amssymb}

\usepackage[mathscr]{eucal}
\usepackage{enumerate}
\usepackage[dvips]{epsfig}

\usepackage{float}  

\usepackage{ifthen} 
\usepackage{extarrows}  
\usepackage{stmaryrd}   
\usepackage{arydshln}     

\usepackage{multirow}
\usepackage{mathtools}
\usepackage[dvips]{geometry} 
\usepackage{listings}
\usepackage{subfig}
\usepackage{dsfont}
\usepackage{xspace}

\usepackage{hyperref}

\theoremstyle{plain}
\newtheorem{thm}{Theorem}[section]

\newtheorem{lma}[thm]{Lemma}
\newtheorem{cor}[thm]{Corollary}

\theoremstyle{definition}

\theoremstyle{remark}
\newtheorem{remark}{Remark}[section]

\newcommand{\Prob}{\operatorname{P}}

\newcommand{\tr}{\operatorname{tr}}

\newcommand{\diag}{\operatorname{diag}}

\newcommand{\Oo}{\mathcal{O}}

\newcommand{\w}{w}	

\newcommand{\AQD}{\frac{h}{2\pi} \sum_{r=1}^{D} \frac{1}{r} H_K \Big(U_{r}^Q \otimes \Big( Z_{r}^Q
-\sqrt{\frac{2}{h}} \Delta \w_h^Q \Big)
 -\Big(Z_{r}^Q - \sqrt{\frac{2}{h}} \Delta \w_h^Q \Big) \otimes U_r^Q \Big)}
 
\newcommand{\ADV}{\frac{h}{2\pi} \sum_{r=1}^{D} \frac{1}{r} H_K \Big(U_{r} \otimes \Big( Z_{r}
-\sqrt{2} V \Big)
 -\Big(Z_{r} - \sqrt{2} V \Big) \otimes U_r \Big)}

\newcommand{\AQ}{\frac{h}{2\pi} \sum_{r=1}^{\infty} \frac{1}{r} H_K \Big(U_{r}^Q \otimes \Big(Z_{r}^Q
 -\sqrt{\frac{2}{h}} \Delta \w_h^Q \Big)
 -\Big(Z_{r}^Q - \sqrt{\frac{2}{h}} \Delta \w_h^Q \Big) \otimes U_r^Q \Big)}

\newcommand{\RQ}{\frac{h}{2\pi} \sum_{r=D+1}^{\infty} \frac{1}{r} H_K \Big( U_{r}^Q \otimes \Big( Z_{r}^Q
-\sqrt{\frac{2}{h}} \Delta \w_h^Q \Big)
 -\Big(Z_{r}^Q - \sqrt{\frac{2}{h}} \Delta \w_h^Q \Big) \otimes U_r^Q \Big)}

\newcommand{\VQr}{U_{r}^Q \otimes \Big( Z_{r}^Q
-\sqrt{\frac{2}{h}} \Delta \w_h^Q \Big)
 -\Big(Z_{r}^Q - \sqrt{\frac{2}{h}} \Delta \w_h^Q \Big) \otimes U_r^Q }

 \newcommand{\SinftyQ}{{\Sigma^Q_{\infty}}}
 \newcommand{\SigmaDQ}{\Sigma^{Q,(D)}}
 \newcommand{\SigmaCondQ}{\Sigma^Q(V_r^Q)_{|Z_r^Q, \Delta\w_h^Q}}
 \newcommand{\Sinfty}{\Sigma_{\infty}^I}

\newcommand{\SigmaQCond}{\Sigma^Q(V_r^Q)_{|Z_r^Q, \Delta\w_h^Q}}
\newcommand{\SigmaICond}{\Sigma^I(V_r^I)_{|Z_r^I, \Delta\w_h^I}}

\newcommand{\SQinfty}{\Sigma^Q_{\infty}}

\allowdisplaybreaks

\title{{Iterated Stochastic Integrals in Infinite Dimensions -- Approximation
and Error Estimates}}

\author{Claudine Leonhard$^1$\thanks{The author was supported partially by
the Graduate School for Computing in Medicine and Life Sciences funded by
Germany's Excellence Initiative [DFG GSC 235/2]. In addition,
this project was funded partially by the Cluster of Excellence ``The Future Ocean''. 
``The Future Ocean'' is funded within the framework
of the Excellence Initiative by the Deutsche Forschungsgemeinschaft (DFG) 
on behalf of the German federal and state governments.
e-mail: cleo@informatik.uni-kiel.de}
\ \ and Andreas R\"o\ss ler$^2$\thanks{e-mail: roessler@math.uni-luebeck.de}
\bigskip
\\
\small{$^1$Department of Computer Science, Christian-Albrechts-Universit\"at zu Kiel,} \\
\small{Christian-Albrechts-Platz 4, 24118 Kiel, Germany} \\[0.2cm]
\small{$^2$Institute of Mathematics, Universit\"at zu L\"ubeck,} \\
\small{Ratzeburger Allee 160, 23562 L\"ubeck, Germany} 
}

\date{}

\begin{document}

\maketitle

\begin{abstract}
Higher order numerical schemes for stochastic partial differential equations that do not possess 
commutative noise require the simulation of iterated stochastic integrals.
In this work, we extend the algorithms derived by Kloeden, Platen, 
and Wright \cite{MR1178485} and by Wiktorsson \cite{MR1843055} 
for the approximation of two-times iterated stochastic integrals involved in
numerical schemes for finite dimensional stochastic ordinary differential equations
to an infinite dimensional setting. 
These methods clear the way for new types of approximation schemes for SPDEs
without commutative noise.
Precisely, we analyze 
two algorithms to approximate two-times iterated integrals with respect to an 
infinite dimensional $Q$-Wiener process in case of a trace class operator $Q$
given the increments of the $Q$-Wiener process. 
Error estimates in the mean-square sense are derived and discussed for both methods.
In contrast to the finite dimensional setting, which is contained as a special case,
the optimal approximation algorithm 
cannot be uniquely determined but is dependent on the covariance operator $Q$. 
This difference arises as the stochastic process is of infinite dimension. 
\end{abstract}
%
%
\section{Introduction}
In order to obtain higher convergence rates in the approximation of stochastic differential equations,
in general, we need to incorporate the information contained in iterated integrals. 
However, in general these integrals can not be simulated directly. Therefore,
we need to replace these terms by an approximation. We illustrate this statement in a finite 
dimensional setting first, although, we are concerned about the approximation of iterated It\^{o} 
integrals in infinite dimensions in this work. \\ \\
In the numerical approximation of stochastic ordinary differential equations (SODEs) that do not
possess commutative noise, iterated stochastic integrals
have to be simulated to achieve a high order of convergence, see~\cite{MR1214374}, \cite{MR1843055}.
One example of such a higher order scheme is the Milstein scheme developed in \cite{MR1335454},
which we present below to illustrate the issue.
For some fixed $d, K\in\mathbb{N}$, we consider a $d$-dimensional SODE of type 
\begin{equation}\label{SODE}
\mathrm{d}X_t = a(X_t) \, \mathrm{d}t + \sum_{j=1}^K b^j(X_t)\, \mathrm{d} \beta_t^j
\end{equation}
with functions $a \colon \mathbb{R}^d \rightarrow \mathbb{R}^d$, 
$b^j=(b^{1,j},\ldots,b^{d,j})^T \colon \mathbb{R}^d \rightarrow \mathbb{R}^d$, $j\in\{1,\ldots,K\}$,
for all $t\geq 0$ and initial value $X_0=x_0 \in \mathbb{R}^d$. 
Moreover, $(\beta^j_t)_{t\geq 0}$, $j \in\{1,\ldots,K\}$,
denote independent real-valued Brownian motions.
For some $T>0$, we divide the time interval $[0,T]$ into $M\in\mathbb{N}$ equal 
time steps $h = \frac{T}{M}$ and 
denote $t_m =mh$ for $m\in\{0,\ldots,M\}$. The increments of the 
Brownian motion are given as
$\Delta \beta^j_m = \beta^j_{t_{m+1}}-\beta^j_{t_m}$
for all  $j\in\{1,\ldots,K\}$ and $m\in\{0,\ldots,M-1\}$.
Then, the Milstein scheme \cite{MR1335454} reads as $Y_0=x_0$ and 
\begin{align*}
    Y_{m+1} &= Y_m + a(Y_m) h + \sum_{j=1}^K b^j(Y_m) \Delta \beta^j_m 
    + \sum_{i,j =1}^K 
    \Big(\frac{\partial b^{l,i}}{\partial x_k}(Y_m)\Big)_{1\leq l,k\leq d}b^{j}(Y_m)
    \int_{t_m}^{t_{m+1}}\int_{t_m}^s  \,\mathrm{d}\beta_r^i \,\mathrm{d}\beta_s^j
\end{align*}
for $m\in\{0,\ldots,M-1\}$ using the notation $Y_m=Y_{t_m}$. 
Under suitable assumptions, 
we obtain the following error estimate 
\begin{equation}\label{ErrorMilstein}
\big( \mathrm{E}\big[|X_T-Y_M|^2\big]\big)^{\frac{1}{2}} \leq Ch,
\end{equation}
see~\cite{MR1214374}.
If SODE \eqref{SODE} does not possess commutative noise, see \cite{MR1214374} for details, 
the Milstein scheme cannot be simplified and 
one has to approximate the iterated stochastic integrals involved in
the method.
We denote these iterated It\^{o} integrals by
\begin{equation*}
    I_{(i,j)}(h) =  I_{(i,j)}(t,t+h) := \int_{t}^{t+h} \int_{t}^s 
    \mathrm{d}\beta_r^i \, \mathrm{d}\beta_s^j
\end{equation*}
for some $t\geq 0$, $h>0$, and for all $i,j  \in\{1,\ldots,K\}$, where $K\in\mathbb{N}$ 
is the number of independent Brownian motions driving the SODE.
The research by Kloeden, Platen and Wright \cite{MR1178485} and by Wiktorsson \cite{MR1843055}
suggests different methods for an approximation of these integrals,
the main ideas are outlined in Section \ref{Sec:Approx}.
We denote by $\bar{I}^{(D)}_{(i,j)}(h)$ the approximation of $I_{(i,j)}(h)$
with the algorithm derived in \cite{MR1178485} 
for $i,j  \in\{1,\ldots,K\}$, $D,K\in\mathbb{N}$, $h>0$.
In \cite{MR1178485}, the authors proved that for all
$i,j \in\{1,\ldots,K\}$ and $h>0$, it holds
\begin{equation}\label{ErrDoubleKP}
    \mathrm{E}\Big[\big\vert I_{(i,j)}(h) - \bar{I}^{(D)}_{(i,j)}(h)\big\vert^2\Big] \leq C\frac{h^2}{D},
\end{equation}
where $D\in\mathbb{N}$ denotes the index of the summand at which the series representation 
of the stochastic double integral
is truncated to obtain the approximation $\bar{I}^{(D)}_{(i,j)}(h)$.
If we use the algorithm derived in \cite{MR1843055} instead, we denote the approximation
of ${I}_{(i,j)}(h)$
by $\hat{I}^{(D)}_{(i,j)}(h)$ for all $i,j  \in\{1,\ldots,K\}$, $h>0$.
This scheme employs the same series representation as proposed in \cite{MR1178485}
but incorporates an approximation of the truncated term additionally.
The error resulting from this scheme is estimated as
\begin{equation}\label{ErrDoubleW}
 \sum_{\substack{i,j=1 \\ i<j}}^K \mathrm{E}\Big[\big\vert I_{(i,j)}(h)
 -\hat{I}^{(D)}_{(i,j)}(h)\big\vert^2\Big]
 \leq \frac{5h^2}{24\pi^2D^2}K^2(K-1),
\end{equation}
where $D$ is again the index of the summand at which the series is truncated to obtain the approximation
and $K$ is the number of independent Brownian motions, see~\cite{MR1843055}.
For fixed $h$ and $K$, both approximations converge in 
the mean-square sense as $D$ goes to infinity - with a different
order of convergence, however.
In the numerical approximation of SODEs, the integer $D$ is determined 
such that the overall order
of convergence in the time step is not distorted. For the Milstein scheme, 
for example, 
error estimate \eqref{ErrorMilstein} is considered, that is, 
a strong order of convergence of 1 
can be achieved. Therefore,
$D\geq \frac{C}{h}$ is chosen for the method derived in \cite{MR1178485}, whereas
$D\geq \frac{\sqrt{5K^2(K-1)}}{\sqrt{24\pi^2h}}$
is selected for the algorithm developed in \cite{MR1843055}, see also
\cite[Cor. 10.6.5]{MR1214374}.
This shows that if we decrease the step size $h$, the value for $D$ 
has to increase faster for the scheme developed
in \cite{MR1178485}.
Note that the error estimate \eqref{ErrDoubleW} depends on the 
number of Brownian motions $K$ as well.
As this number is fixed in the setting of finite dimensional SODEs, 
this factor is not crucial but simply a constant. 
Therefore, the algorithm proposed by Wiktorsson \cite{MR1843055} 
is superior to the one derived in 
\cite{MR1178485} in terms of the computational effort when a given 
order of convergence in the step size $h$
is to be achieved. \\ \\
The same issue arises in the context of higher order numerical 
schemes designed for infinite dimensional stochastic differential equations 
that need not have commutative noise. There, we also have to
approximate the involved iterated stochastic integrals in 
order to implement the scheme. This time, however, 
the stochastic process is infinite dimensional, in general.
In this work, we aim at devising numerical algorithms for the simulation of iterated integrals
which arise, for example, in the approximation of the mild solution
of stochastic partial differential 
equations (SPDEs) of type
\begin{equation}\label{SPDE}
\mathrm{d} X_t = \big( AX_t+F(X_t)\big) \, \mathrm{d}t + B(X_t)\, \mathrm{d}W_t, 
\quad t\in(0,T], \quad X_0 = \xi, 
\end{equation}
where the commutativity condition from \cite{MR3320928}
\begin{equation}\label{CommutativityCond}
B'(y)\big(B(y)u,v\big)= B'(y)\big(B(y)v,u\big)
\end{equation}
for all $y\in H_{\beta}$, $u,v\in U_0$ is \emph{not} assumed to hold.
Here, $H_{\beta}=D((-A)^{\beta})$ denotes a separable Hilbert space for some $\beta \in[0,1)$.
The operators $A$, $F$, $B$,
and the initial value $\xi$
are assumed to fulfill the conditions imposed for the existence of a unique mild
solution, see~\cite{MR3236753},
and are not specified further.
The spaces are introduced in Section~\ref{Sec:Approx} and 
$(W_t)_{t\geq 0}$ denotes a 
$Q$-Wiener process taking values in some separable Hilbert space $U$
for some trace class operator $Q$.
In order to approximate the mild solution of SPDEs of type \eqref{SPDE}
with a higher order scheme,
we need to simulate iterated stochastic integrals of the form
\begin{equation}\label{DoubleIntSPDE}
    \int_t^{t+h} \Psi\left(\int_t^s \Phi \, \mathrm{d}W_r\right)\, 
    \mathrm{d}W_s,
\end{equation}
for $t \geq 0$, $h>0$, and some operators $\Psi$, $\Phi$ specified in Section~\ref{Sec:Approx}. 
These terms arise if condition \eqref{CommutativityCond}
is not fulfilled, for example,
in the Milstein scheme for SPDEs \cite{MR3320928}.
%
In this Milstein scheme, it holds $\Psi = B'(Y_t)$ and $\Phi = B(Y_t)$ for some
$B \colon H \rightarrow L_{HS}(U_0,H)$ and an approximation
$Y_t \in H_{\beta}$ with $t\geq 0$ and $\beta \in[0,1)$, where
$L_{HS}(U_0,H)$ denotes the space of all Hilbert-Schmidt operators from $U_0$
to $H$. For more details, we refer to~\cite{MR3320928}. \\ \\
We want to emphasize that the algorithms developed for 
the approximation of iterated stochastic
integrals in the setting of SODEs are designed for some fixed 
finite number $K$ of driving Brownian motions
and that the approximation error \eqref{ErrDoubleW} even 
involves this number $K$ as a constant.
In contrast, when approximating the solution of SPDEs driven by 
an infinite dimensional $Q$-Wiener process, 
this number corresponds to the dimension of the finite-dimensional 
approximation subspace where the $Q$-Wiener process is projected in.
Thus, the dimension $K$ of the approximation subspace has to increase, in general, 
to attain higher accuracy, i.e., $K$ is not constant anymore but has to 
increase as well; see the error estimate of the
Milstein scheme for SPDEs in \cite{MR3320928}, for example.
Therefore, this aspect has to be taken into account in order to 
identify an appropriate approximation algorithm.
In the following, we derive two algorithms for the approximation of iterated integrals of type
\eqref{DoubleIntSPDE} based on the methods developed for the finite dimensional setting
by Kloeden, Platen, and Wright \cite{MR1178485} and by Wiktorsson \cite{MR1843055}.
These algorithms allow for the first time to implement higher order schemes
for SPDEs that do not possess commutative noise and include the algorithms
that can be used for finite dimensional SODEs as a special case.
We show that the algorithm that is superior in the setting of an infinite dimensional 
$Q$-Wiener process cannot be uniquely determined in general
but is dependent on the covariance operator $Q$.
In the analysis of the approximation error,
we need to incorporate the eigenvalues of the covariance operator $Q$. 
For the algorithm based on the approach by Kloeden, Platen, and Wright \cite{MR1178485},
we obtain a similar estimate as in \eqref{ErrDoubleKP} in 
the mean-square sense, see Corollary~\ref{Algo1Lemma}.
For the method derived in the work of 
Wiktorsson \cite{MR1843055}, we can prove two differing error estimates for 
the case of infinite dimensional $Q$-Wiener processes by different means.
One is the same, apart from constants, as estimate \eqref{ErrDoubleW}.
Moreover, the fact that we integrate with respect
to a $Q$-Wiener process with a trace class operator $Q$ allows for an alternative proof to
the one given in \cite{MR1843055}. The result allows a possibly
superior convergence in $K$ - this depends on the rate of decay of the eigenvalues of 
$Q$. Details can be found in Theorem~\ref{Algo2} and Theorem~\ref{Algo2Alternative}.
\section{Approximation of Iterated Stochastic Integrals}\label{Sec:Approx}
Throughout this work, we fix the following setting and notation.
Let $H$ and $U$ be separable real-valued Hilbert spaces. 
In the following, let $(\Omega,\mathcal{F},P,(\mathcal{F}_t)_{t\geq 0})$ 
be a probability space, let $(W_t)_{t \geq 0}$ denote a $U$-valued 
$Q$-Wiener process with respect to $(\mathcal{F}_t)_{t\geq 0}$
where $Q$ is a trace class covariance operator,
and let $U_0 := Q^{{1}/{2}}U$.
We define
$L(U,H)_{U_0} :=\{T|_{U_0} : T\in L(U,H)\}$ which is a dense subset of the space
of Hilbert-Schmidt operators $L_{HS}(U_0,H)$ \cite{MR2329435}.
Moreover, we assume that the operators $\Phi$ and $\Psi$ in \eqref{DoubleIntSPDE} fulfill 
\begin{itemize}
  \item[(A1)] $\Phi \in L(U,H)_{U_0}$ with $\|\Phi Q^{-\alpha}\|_{L_{HS}(U_0,H)}<C$,
  \item[(A2)] $\Psi\in L(H,L(Q^{-\alpha}U,H)_{U_0})$
\end{itemize} 
for some $\alpha\in(0,\infty)$. The parameter $\alpha$ determines the rate of convergence 
for the approximation of the $Q$-Wiener process, see Theorem~\ref{Algo1} or Theorem~\ref{Algo2}.
Note that assumption (A1), needed to prove the convergence of 
the approximation algorithms for iterated 
integrals in Theorem~\ref{Algo1}, Theorem~\ref{Algo2}, and Theorem~\ref{Algo2Alternative},
is less restrictive
than the condition imposed on the operator $B$ in SPDE \eqref{SPDE} 
to obtain the error estimate 
for some numerical scheme to approximate its mild solution, 
e.g., in \cite{2015arXiv150908427L}. 
However, for the Milstein scheme in \cite{MR3320928},
assumption (A2) does not need to be 
fulfilled for the error analysis of the Milstein scheme to hold true. \\ \\
If we are interested in 
the approximation of, for example, the mild solution of \eqref{SPDE}, 
a combination of the error estimate for a numerical scheme
to obtain this process
and the error from the approximation of the iterated integrals has to be analyzed.
In this case, we impose the following assumptions instead 
\begin{itemize}
  \item[(B1)] $\Phi \in L(U,H)_{U_0}$, 
  \item[(B2)] $\Psi \in L(H,L(U,H)_{U_0})$.
\end{itemize}
For the convergence results in this case, we refer to  
Corollary~\ref{Algo1Lemma} and Corollary~\ref{Algo2Lemma}, which have to be combined with estimates
on the respective numerical scheme.
These weaker conditions are sufficient as in the proof
the $Q$-Wiener process is approximated before the iterated 
integral is compared to the approximation. \\ \\
Let $Q \in L(U)$ be a nonnegative and symmetric trace class operator with 
eigenvalues $\eta_j$ and corresponding eigenfunctions $\tilde{e}_j$ for 
$j \in \mathcal{J}$ where $\mathcal{J}$ is some countable index set.
The eigenfunctions $\{\tilde{e}_j, j\in\mathcal{J}\}$ constitute an 
orthonormal basis of $U$, see~\cite[Prop. 2.1.5]{MR2329435}. 
Then, for the $Q$-Wiener process $(W_t)_{t \geq 0}$, the following series 
representation holds, see~\cite[Prop. 2.1.10]{MR2329435},
\begin{equation}\label{QSeries}
    W_t = \sum_{j \in \mathcal{J}} \sqrt{\eta_j} \beta_t^j \tilde{e}_j, 
    \quad t\geq 0.
\end{equation}
Here, $(\beta^j_t)_{t\geq 0}$ with $j\in\{k\in\mathcal{J} \, |\, \eta_k>0 \}$ 
are independent real-valued Brownian motions.
As the $Q$-Wiener process $(W_t)_{t\geq 0}$ is an infinite dimensional stochastic process,
it has to be projected to some finite dimensional subspace 
by truncating the series \eqref{QSeries} such that it can be simulated 
in a numerical scheme.
For $K \in \mathbb{N}$, we denote by $(W_t^K)_{t\geq 0}$ the projected
$Q$-Wiener process, which is defined as
\begin{equation}\label{QSeriesK}
    W_t^K = \sum_{j\in\mathcal{J}_K} 
    \langle W_t, \tilde{e}_j \rangle_U \, \tilde{e}_j , \quad t\geq 0,
\end{equation}
for some finite index set $\mathcal{J}_K \subset \mathcal{J}$ with $|\mathcal{J}_K|=K$.
This expression allows to write the iterated integral with respect to the projected
$Q$-Wiener process $(W_t^K)_{t\geq 0}$ for any $t\geq0$ and $h>0$ as
\begin{align*}
    \int_t^{t+h} \Psi \bigg( \int_t^s \Phi \,  \mathrm{d}W_r^K \bigg) \, \mathrm{d}W_s^K
    &= \int_t^{t+h} \Psi \bigg( \int_t^s \Phi \sum_{i \in \mathcal{J}_K} 
    \langle \mathrm{d}W_r, \tilde{e}_i \rangle_U \tilde{e}_i \bigg)  
    \sum_{j \in \mathcal{J}_K} \langle \mathrm{d}W_s, \tilde{e}_j \rangle_U \tilde{e}_j \\
    & = \sum_{i,j \in \mathcal{J}_K} I^Q_{(i,j)}(t,t+h)
    \, \Psi\big(\Phi\tilde{e}_i, \tilde{e}_j\big)
%
\end{align*}
with 
\begin{equation*}
  I^Q_{(i,j)}(t,t+h) := \int_{t}^{t+h} \int_{t}^s 
  \langle \mathrm{d} W_r, \tilde{e}_i \rangle_U \, 
  \langle \mathrm{d} W_s, \tilde{e}_j \rangle_U
\end{equation*}
for $i,j  \in \mathcal{J}_K$. 
Therefore, we aim at devising a method to approximate the iterated stochastic integrals
$I^Q_{(i,j)}(t,t+h)$ for all $i,j\in\mathcal{J}_K$. 
Below, we introduce two such algorithms and analyze 
as well as discuss their convergence properties. 
For simplicity of notation, we assume, without loss of generality, 
$\mathcal{J}_K = \{1,2,\ldots,K\}$ with $\eta_j \neq 0$ for $j \in \mathcal{J}_K$ 
and denote $I^Q_{(i,j)}(h)=I^Q_{(i,j)}(t,t+h)$ in the following.
\subsection{Algorithm~1}
In the following, we mainly adapt the method introduced by Kloeden, Platen, 
and Wright~\cite{MR1178485} to the setting of infinite dimensional stochastic processes.
Here, we additionally have to take into account the error arising from the projection
of the $Q$-Wiener process to a finite dimensional subspace.
%
\\ \\
For some $t \geq 0$, the coefficients of the projected $Q$-Wiener process 
$\w_{t}^j := \langle W_t, \tilde{e}_j \rangle_U$
are independent real valued random variables that are
$N(0, \eta_j \, t)$ distributed for $j \in \mathcal{J}$. 
Thus, the increments $\Delta \w_{h}^j := 
\langle W_{t+h}-W_t, \tilde{e}_j \rangle_U$ can be easily simulated 
since $\Delta \w_{h}^j$ is $N(0, \eta_j \, h)$ distributed for 
$j \in \mathcal{J}$ and $h \geq 0$.
Our goal is to obtain an approximation of the iterated integrals $I^Q_{(i,j)}(h)$ for all
$i,j \in \mathcal{J}_K$, $K\in\mathbb{N}$, $h>0$
given the realizations of the increments $\Delta \w_{h}^j$ for $j \in \mathcal{J}_K$. 
The following derivation of the approximation method follows the
representation in \cite{MR1178485} closely.
Below, let $K \in \mathbb{N}$ be arbitrarily fixed and 
let us introduce the scaled Brownian bridge process 
$(\w_s^j-\frac{s}{h}\w_h^j)_{0\leq s\leq h}$ for
$j \in \mathcal{J}_K$ and some $h \in (0,T]$. 
We consider its series expansion 
\begin{align}\label{FourierBBridge}
  \w_s^j -\frac{s}{h}\w_h^j = \frac{1}{2} a_0^j
  +\sum_{r=1}^{\infty} \Big(a^j_r\cos\Big(\frac{2r\pi s}{h}\Big) 
  +b^j_r \sin\Big(\frac{2r\pi s}{h}\Big)\Big)
\end{align}
which converges in $L^2(\Omega)$.
The coefficients are given by the following expressions
\begin{align*}
  a^j_r = \frac{2}{h} \int_0^h (\w_u^j-\frac{u}{h} \w_h^j)
  \cos\Big(\frac{2r\pi u}{h}\Big)\,\mathrm{d}u, 
  \quad
  b^j_r = \frac{2}{h} \int_0^h (\w_u^j-\frac{u}{h} \w_h^j)
  \sin\Big(\frac{2r\pi u}{h}\Big)\,\mathrm{d}u
\end{align*}
for all $j  \in \mathcal{J}_K$, $r \in\mathbb{N}_0$,
and all $0\leq s\leq h \leq T$,
see also \cite{MR1178485}. 
All coefficients $a^j_r$ and $b^j_r$ are independent and
$N(0,\tfrac{\eta_j h}{2 \pi^2 r^2})$ distributed for $r \in \mathbb{N}$ 
and $j \in \mathcal{J}_K$ and it holds $a_0^j = -2 \sum_{r=1}^{\infty} a_r^j$.
In contrast to \cite{MR1178485}, the distributions
of the coefficients additionally depend on the eigenvalues $\eta_j$
of the covariance operator $Q$.
In order to obtain an approximation of the
scaled Brownian motion $(\w_s^j)_{0\leq s\leq h}$
for some $h \in (0,T]$,
we truncate expression \eqref{FourierBBridge} at some integer $R\in\mathbb{N}$ and define
\begin{equation}\label{BrownianBridgeApprox}
  {\w_s^j}^{(R)} =\frac{s}{h} \w_h^j + \frac{1}{2} a_0^j +\sum_{r=1}^{R}
  \Big(a^j_r\cos\Big(\frac{2r\pi s}{h}\Big) +b^j_r \sin\Big(\frac{2r\pi s}{h}\Big)\Big).
\end{equation}
In fact, we are interested in the integration with respect to this process. 
According to Wong and Zakai \cite{MR0195142,MR0183023},
or \cite[Ch. 6.1]{MR1214374}, an integral with respect to process 
\eqref{BrownianBridgeApprox} converges to a Stratonovich
integral $J(h)$ as $R\rightarrow \infty$. We are, however, 
interested in the It\^{o} stochastic integral.
Following \cite[p.~174]{MR1214374}, the Stratonovich integral $J^Q_{(i,j)}(h)$ 
can be converted to an It\^{o}
integral $I^Q_{(i,j)}(h)$,
$i,j \in \mathcal{J}_K$, according to
\begin{equation*}
  {I}^Q_{(i,j)}(h) = {J}^Q_{(i,j)}(h) -\frac{1}{2} \, h \, \eta_i \, \mathds{1}_{i=j}.
\end{equation*}
That is, ${I}^Q_{(i,j)}(h) = {J}^Q_{(i,j)}(h)$ for all 
$i,j \in \mathcal{J}_K$ with $i \neq j$.
Moreover, we compute
\begin{equation*}
  I^Q_{(i,i)}(h) = \frac{\big(\Delta \w^i_h)^2 - \eta_i \, h}{2}
\end{equation*}
directly for $ i \in \mathcal{J}_K$, see~\cite[p.~171]{MR1214374}.
This implies that we only have to approximate ${I}^Q_{(i,j)}(h)$
for $i,j \in \mathcal{J}_K$ with $i \neq j$.
Thus, we obtain the desired approximation of the It\^{o} stochastic
integral directly by integrating with respect
to process \eqref{BrownianBridgeApprox}.
Without loss of generality let $t=0$.
By \eqref{FourierBBridge}, we obtain the following expression for 
the iterated stochastic integral
\begin{align}\label{DoubleLevy}
	I^Q_{(i,j)}(h) 
	&= \int_0^h  \w^i_u \, \mathrm{d}\w^j_u \nonumber \\
	&= \int_0^h \bigg( \frac{u}{h} \w_h^i + \frac{1}{2} a^i_0
	+ \sum_{r=1}^{\infty} \Big( a^i_r \cos \Big( \frac{2r \pi u}{h} \Big) 
   	+ b^i_r \sin \Big( \frac{2r \pi u}{h} \Big) \Big) \bigg) \, \mathrm{d}\w^j_u \nonumber \\
  	&= \frac{ \w_h^i}{h} \int_0^h u \, \mathrm{d}\w^j_u  
  	+ \frac{1}{2} a^i_0 \w^j_h \nonumber  \\
  	& \quad + \sum_{r=1}^{\infty} \Big( a^i_r \Big( \w_h^j 
  	+ \int_0^h \frac{2r \pi}{h}
  	\sin \Big( \frac{2r \pi u}{h} \Big) \w_u^j \, \mathrm{d}u \Big)
  	-b^i_r \int_0^h  \frac{2r \pi}{h} 
  	\cos \Big( \frac{2r \pi u}{h} \Big) \w_u^j \, \mathrm{d}u \Big) \nonumber \\
   	&= \frac{1}{2} \w_h^i \w_h^j
	- \frac{1}{2} (a^j_0 \w_h^i - a^i_0 \w^j_h)  
	+ \sum_{r=1}^{\infty} \Big( a^i_r \Big( \w_h^j 
	+ r \pi \Big( b_r^j - \frac{\w_h^j}{r \pi} \Big) \Big)
  	- \frac{2r \pi}{h} b^i_r \frac{h}{2} a_r^j \Big) \nonumber \\
   	&= \frac{1}{2} \w_h^i \w_h^j
   	- \frac{1}{2} ( a^j_0 \w_h^i - a^i_0 \w^j_h)  
   	+ \pi \sum_{r=1}^{\infty} r (a^i_r b^j_r - b^i_r a^j_r ) \nonumber \\
   	&= \frac{1}{2} \Delta\w_h^i \Delta\w_h^j 
   	+ \pi \sum_{r=1}^{\infty} r \Big(a^i_r \Big( b^j_r - \frac{1}{\pi r} \Delta \w^j_h \Big) 
   	- \Big( b^i_r - \frac{1}{\pi r} \Delta \w^i_h \Big) a^j_r \Big)
\end{align}
for all $i,j \in \mathcal{J}_K$, $i \neq j$, and $h>0$.  
Here, we employed the fact that 
$\int_0^h f(u) \, \mathrm{d} \w_u^j = f(h) \w_h^j - \int_0^h f'(u) \w_u^j \, \mathrm{d}u$ for
a continuously differentiable function $f \colon [0,h] \rightarrow \mathbb{R}$, 
$h > 0$, see~\cite[p.~89]{MR1214374},
$a_0^j = \frac{2}{h} \int_0^h \w_u^j \, \mathrm{d}u -\w_h^j$, and 
especially $a_0^j = -2 \sum_{r=1}^{\infty} a_r^j$,
$j \in \mathcal{J}_K$.
Expression \eqref{DoubleLevy} involves some scaled L\'{e}vy stochastic area
integrals which are defined as
\begin{equation}\label{AreaIntKP}
	A_{(i,j)}^Q(h) := 
 	\pi \sum_{r=1}^{\infty} r \Big(a^i_r \Big( b^j_r - \frac{1}{\pi r} \Delta \w^j_h \Big) 
   	- \Big( b^i_r - \frac{1}{\pi r} \Delta \w^i_h \Big) a^j_r \Big)
\end{equation}
for all $ i,j \in \mathcal{J}_K$, $i \neq j$, $h>0$.
We approximate these terms instead of the iterated stochastic integrals, as proposed in
\cite{MR1178485} and \cite{MR1843055}.
Due to the relations
\begin{align}
	I_{(i,j)}^Q(h) &= \frac{\Delta \w_h^i \ \Delta\w_h^j 
	- h \, \eta_i \, \delta_{ij}}{2} + A_{(i,j)}^Q(h) \label{AandI1} \\
	A^Q_{(j,i)}(h) &= -A^Q_{(i,j)}(h) \label{AandI2}\\
	A^Q_{(i,i)}(h)& = 0 \label{AandI3}
\end{align}
$\Prob$-a.s.\ for all $i,j  \in \mathcal{J}_K$, $h>0$, see~\cite{MR1843055}, it is sufficient to
simulate $A_{(i,j)}^Q(h)$  for $i,j \in \mathcal{J}_K$ with $i<j$.
By the distributional properties
of $a_r^i$ and $b_r^i$ for $r \in \mathbb{N}_0$, $i \in \mathcal{J}_K$,
we write 
\begin{equation*} 
	A_{(i,j)}^Q(h) = \frac{h}{2\pi} \sum_{r=1}^{\infty} \frac{1}{r} \Big(U_{ri}^Q 
	\Big( Z_{rj}^Q - \sqrt{\frac{2}{h}} \Delta \w_h^j \Big)
 	- U_{rj}^Q \Big( Z_{ri}^Q - \sqrt{\frac{2}{h}} \Delta \w_h^i \Big) \Big)
\end{equation*}
for all $ i,j \in \mathcal{J}_K$, $i \neq j$, $h>0$
and $A^Q(h) = \big( A_{(i,j)}^Q(h)\big)_{1 \leq i, j \leq K}$ in order to relate to 
the derivation in \cite{MR1843055}.
This representation entails the random variables $U_{ri}^Q \sim N(0,\eta_i)$, $Z_{ri}^Q \sim N(0,\eta_i)$, and
$\Delta \w_h^i \sim N(0,\eta_i \, h)$ that are all independent for $i \in \mathcal{J}_K$,
$r \in\mathbb{N}$.
As described above, we only need to approximate  $A_{(i,j)}^{Q}(h)$, $h>0$, 
for $i, j \in \mathcal{J}_K$ with $i <j $, that is, we want to simulate
\begin{equation*}
 	\tilde{A}^Q(h) =(A_{1,2}^Q(h),\ldots,A_{1,K}^Q(h),A_{2,3}^Q(h),\ldots,A_{2,K}^Q(h),
 	\ldots,A_{l,l+1}^Q(h),\ldots,A_{l,K}^Q(h),\ldots,A_{K-1,K}^Q(h)).
\end{equation*}
Therefore, we write
\begin{equation*}
	\text{vec}(A^Q(h)^T) = (A^Q_{1,1}(h),\ldots,A^Q_{1,K}(h),A^Q_{2,1}(h),\ldots,A^Q_{2,K}(h),\ldots,
	A^Q_{K,1}(h),\ldots,A^Q_{K,K}(h))^T
\end{equation*}
and introduce the selection matrix 
\begin{equation}\label{SelectionMatrix}
	H_K = \begin{pmatrix}
	0_{K-1\times 1} & I_{K-1} & 0_{K-1\times K(K-1)}\\
	0_{K-2\times K+2} & I_{K-2} & 0_{K-2\times K(K-2)}\\
	\vdots   & \vdots & \vdots\\
	0_{K-l\times(l-1)K+l} & I_{K-l} & 0_{K-l\times K(K-l)}\\
	\vdots   & \vdots & \vdots\\
	0_{1\times(K-2)K+K-1} & 1 & 0_{1\times K}
\end{pmatrix}
\end{equation}
which defines the integrals that have to be computed, compare to \cite{MR1843055}.
Further, we define the matrix
\begin{equation*}
	Q_K := \diag({\eta_1}, \ldots, {\eta_K}) .
\end{equation*}
This allows to express the vector $\tilde{A}^Q(h)$ as
\begin{align}
	\tilde{A}^Q(h) &= H_K \text{vec}(A^Q(h)^T) \nonumber \\
	&= \AQ \label{A_Vector}
\end{align}
with $\Delta\w_h^Q =(\Delta\w_h^1, \dots, \Delta\w_h^K)^T$ and the random vectors
$U_r^Q = (U_{r1}^Q, \ldots, U_{rK}^Q)^T$
and $Z_r^Q = (Z_{r1}^Q, \ldots, Z_{rK}^Q)^T$ that are independent 
and identically $N(0_K,Q_K)$ distributed for all $r \in \mathbb{N}$.
As expression \eqref{A_Vector} contains an infinite sum, we need to 
truncate it in order to compute this vector. 
For some $D \in \mathbb{N}$, this approximation is denoted as
\begin{equation} \label{Alg1-AQD-truncated}
	\tilde{A}^{Q,(D)}(h) := \AQD
\end{equation}
and we specify the remainder 
\begin{equation} \label{Alg1-RQD-truncation-error}
	\tilde{R}^{Q,(D)}(h) := \RQ .
\end{equation}
Let $A^I(h) = Q_K^{-{1}/{2}} A^Q(h) Q_K^{-{1}/{2}}$ denote the
matrix containing the standard L\'{e}vy stochastic area integrals that correspond to the case 
that $Q_K=I_K$, i.e., $\eta_j=1$ for all $j \in \mathcal{J}_K$.
Therewith, we obtain the relationship 
\begin{align*}
	\tilde{A}^Q(h) &=  H_K \text{vec}(A^Q(h)^T)
	= H_K \big( Q_K^{{1}/{2}} \otimes Q_K^{{1}/{2}} \big)
	\text{vec}(A^I(h)^T) \\
	&= H_K \big( Q_K^{{1}/{2}} \otimes Q_K^{{1}/{2}} \big)
	H_K^T H_K \text{vec}(A^I(h)^T) \\
	&= H_K \big( Q_K^{{1}/{2}} \otimes Q_K^{{1}/{2}} \big) H_K^T \tilde{A}^I(h),
\end{align*}
where $\tilde{A}^I(h) := H_K \text{vec}(A^I(h)^T)$ and where we employed
\begin{equation*}
	H_K^T H_K = \diag( 0, \mathbf{1}^T_{K-1}, 0, 0, \mathbf{1}^T_{K-2}, \ldots, \mathbf{0}^T_l,
	\mathbf{1}^T_{K-l}, \ldots, \mathbf{0}^T_{K-1}, 1, \mathbf{0}^T_{K} ) \in \mathbb{R}^{K^2 \times K^2}
\end{equation*}
and the fact that we are interested in indices $i,j \in \mathcal{J}_K$ with $i<j$ only. We denote
\begin{equation*}
 \tilde{Q}_K := H_K \big(Q_K^{{1}/{2}} \otimes Q_K^{{1}/{2}}\big)H_K^T,
\end{equation*}
which is of size $L \times L$ with $L = \frac{K(K-1)}{2}$, such that the vector of interest
is given by
\begin{equation*}
 \tilde{A}^Q(h) = \tilde{Q}_K  \tilde{A}^I(h).
\end{equation*}
Now, we can represent the approximation $\tilde{A}^{Q,(D)}(h)$ of $\tilde{A}^{Q}(h)$ 
as
\begin{equation*}
 	\tilde{A}^{Q,(D)}(h) = \tilde{Q}_K \tilde{A}^{I,(D)}(h)
\end{equation*}
and the vector of truncation errors by
	$\tilde{R}^{Q,(D)}(h) = \tilde{Q}_K \tilde{R}^{I,(D)}(h)$
where $\tilde{A}^{I,(D)}(h)$ and $\tilde{R}^{I,(D)}(h)$ denote, in analogy to  
\eqref{Alg1-AQD-truncated} and \eqref{Alg1-RQD-truncation-error},
the truncated part of $\tilde{A}^I(h)$ and its truncation error, respectively.
Note that $\tilde{A}^I(h)$ and especially $\tilde{A}^{I,(D)}(h)$ correspond 
to the case where $\eta_j=1$ for all $j \in \mathcal{J}_K$, i.e., 
$Q_K=I_K$ in \eqref{A_Vector} and \eqref{Alg1-AQD-truncated}, respectively.
This also corresponds to the setting in \cite{MR1178485} if $\mathcal{J}$ is finite.\\ \\
\phantomsection \label{Sec:Algo1}\noindent
We summarize the representation above to formulate Algorithm~1
for some $h>0$, $t,t+h \in [0,T]$, and $D,K\in\mathbb{N}$:
\begin{sffamily}
\begin{enumerate}
 \item  For $j \in \mathcal{J}_K$, simulate the Fourier coefficients 
 $\Delta \w_{h}^j = \langle W_{t+h}-W_t, \tilde{e}_j \rangle_U$ 
 of the increment $W_{t+h}-W_t$ 
 with
 $\Delta \w_h^Q = \big( \Delta \w^1_h, \ldots, \Delta \w^K_h \big)^T$ as
 \begin{equation*}
 		\Delta\w_h^Q = \sqrt{h} \, Q_K^{{1}/{2}} V
 \end{equation*}
 where $V  \sim N(0_K,I_K)$.
 \item Approximate $\tilde{A}^Q(h)$ as
  \begin{equation*}
   		\tilde{A}^{Q,(D)}(h) 
    		= H_K \big( Q_K^{{1}/{2}} \otimes Q_K^{{1}/{2}} \big) H_K^T
    		\ADV
 \end{equation*}
 where $U_r, Z_r \sim N(0_K, I_K)$ are independent.
 \item Compute the approximation $ \text{vec}((\bar{I}^{Q,(D)}(h))^T)$ of $ \text{vec}((I^Q(h)^T)$ as
 \begin{equation*}
 		\text{vec}((\bar{I}^{Q,(D)}(h))^T) = \frac{\Delta\w_h^Q \otimes \Delta\w_h^Q - \text{vec}(h \, Q_K)}{2}
 		+ (I_{K^2} - S_K) H_K^T \tilde{A}^{Q,(D)}(h) 
 \end{equation*}
 with $S_K := \sum_{i=1}^K \mathbf{e}_i^T \otimes (I_K \otimes \mathbf{e}_i)$, where
 $\mathbf{e}_i$ denotes the $i$-th unity vector. 
\end{enumerate}
\end{sffamily}
We obtain the following error estimate for this approximation method; 
the mean-square error converges with
order $1/2$ in $D$ while the convergence in $K$
is determined by the operator $Q$.
The first term results from the approximation of the 
$Q$-Wiener process by $(W_t^K)_{t\geq 0}$,
whereas the second term is due to the approximation of 
the iterated integral with respect to this
truncated process by Algorithm~1. 
\begin{thm}[Convergence for Algorithm~1]\label{Algo1}
	Assume that $Q \in L(U)$ is a nonnegative and symmetric trace class operator with eigenvalues 
	$\{ \eta_j : j \in \mathcal{J}\}$.
	Further, let $\Phi \in L(U,H)_{U_0}$ with $\| \Phi Q^{-\alpha}\|_{L_{HS}(U_0,H)}<C$ for some $C>0$,
	let
	$\Psi \in L(H,L(Q^{-\alpha}U,H)_{U_0})$ for some $\alpha\in(0,\infty)$, i.e., (A1) and (A2) 
	are fulfilled, and let $(W_t)_{t \in [0,T]}$ be a $Q$-Wiener process. Then, it holds
 	\begin{align*}
 		&\bigg( \mathrm{E}\bigg[\Big\|\int_t^{t+h} 
 		\Psi\Big( \int_t^s\Phi \, \mathrm{d}W_r\Big) \,\mathrm{d}W_s
 		- \sum_{i, j \in \mathcal{J}_K} \bar{I}_{(i,j)}^{Q,(D)}(h) \;
 		\Psi\big(\Phi\tilde{e}_i, \tilde{e}_j\big)
 		\Big\|_H^2\bigg]\bigg)^{\frac{1}{2}} \\ 
 		&\leq  C_Qh\Big(\sup_{j\in\mathcal{J} \setminus
 		\mathcal{J}_K}\eta_j\Big)^{\alpha} + C_Q\frac{h}{\pi \sqrt{D}}
 	\end{align*}
 	for some $C_Q>0$ and all $h>0$, $t,t+h\in[0,T]$,  $D, K \in \mathbb{N}$,
 	and $\mathcal{J}_K \subset \mathcal{J}$ with $|\mathcal{J}_K|=K$.
\end{thm}
\begin{proof}
 For a proof, we refer to Section~\ref{Sec:Proofs}.
\end{proof}
 Note that in the convergence analysis of numerical schemes for SPDEs, we compare the approximation
 of the iterated stochastic integrals to integrals with respect
 to $(W_t^K)_{t\geq0}$, $K\in\mathbb{N}$, see the proofs in \cite{MR3320928} and 
 \cite{2015arXiv150908427L}, for example, that is, the analysis involves
  the error estimate stated in Corollary~\ref{Algo1Lemma} below.
  We want to emphasize that this estimate is independent of the integer $K$.
%
\begin{cor}\label{Algo1Lemma}
	Assume that $Q$ is a nonnegative and symmetric trace class operator 
	and $(W_t)_{t \geq 0}$ is a $Q$-Wiener process.
 	Furthermore, let $\Phi \in L(U,H)_{U_0}$ and $\Psi \in L(H,L(U,H)_{U_0})$, i.e., assumptions
 	(B1) and (B2) are fulfilled. Then, it holds
 	\begin{align*}
 		&\bigg( \mathrm{E} \bigg[ \Big\| \int_t^{t+h} \Psi
 		\Big( \int_t^s\Phi \, \mathrm{d}W_r^K \Big) \, \mathrm{d}W_s^K
 		- \sum_{i, j \in \mathcal{J}_K} \bar{I}_{(i,j)}^{Q,(D)}(h)
 		\; \Psi\big(\Phi \tilde{e}_i, \tilde{e}_j \big)
 		\Big\|_H^2 \bigg] \bigg)^{\frac{1}{2}} \leq  C_Q \frac{h}{\pi \sqrt{D}}
 	\end{align*}
 	for some $C_Q>0$ and all $h>0$, $t,t+h\in[0,T]$,  $D,K\in\mathbb{N}$,
 	and $\mathcal{J}_K \subset \mathcal{J}$ with $|\mathcal{J}_K|=K$.
\end{cor}
\begin{proof}
	If we set $\eta_i=0$ for all $i \in \mathcal{J} \setminus \mathcal{J}_K$,
	the result follows directly from Theorem~\ref{Algo1}.
\end{proof}
 Next, we outline an alternative algorithm to approximate integrals of type \eqref{DoubleIntSPDE}. 
 In contrast 
 to the method presented above, the vector of tail sums $\tilde{R}^{Q,(D)}(h)$ is approximated and included
 in the computation.
\subsection{Algorithm~2}
The following derivation is based on the scheme developed by Wiktorsson \cite{MR1843055} for SODEs.
In the finite dimensional setting, the error estimate \eqref{ErrDoubleW} depends on the
number of Brownian motions $K$ additionally to the time step size $h$.
This suggests that the computational cost 
involved in the simulation of
the stochastic double
integrals is much larger in the setting of
SPDEs as the number of independent Brownian motions is, in general, not finite,
see also expression \eqref{QSeries}.
The eigenvalues of the $Q$-Wiener process are, however,
not incorporated in the error estimate \eqref{ErrDoubleW}.
For example, if we assume $\eta_j = \Oo(j^{-\rho_Q})$ 
for some $\rho_Q>1$, $C>0$, $j\in\mathcal{J} \subset \mathbb{N}$, we obtain -- for $\rho_Q\in(1,3)$ -- an
improved error estimate which depends on the rate of decay 
of the eigenvalues instead of some fixed exponent of $K$.
This results from the fact that we integrate with respect to a $Q$-Wiener process in our setting,
where $Q$ is a nonnegative, symmetric trace class operator.
For $\rho_Q\geq 3$, we can show that the exponent of $K$ is bounded by 3.\\ \\
As before, we truncate the series \eqref{A_Vector} at some integer $D\in\mathbb{N}$ and obtain
the approximation $\tilde{A}^{Q,(D)}(h)$ in \eqref{Alg1-AQD-truncated}.
The vector of tail sums $\tilde{R}^{Q,(D)}(h)$ in \eqref{Alg1-RQD-truncation-error}, 
however, is not discarded but approximated by a 
multivariate normally distributed random
vector instead, as described in \cite{MR1843055} for $Q_K = I_K$ and $|\mathcal{J}|=K$.
First, we determine the distribution of the tail sums;
for $r\in \mathbb{N}$, we compute the covariance matrix of
\begin{equation*}
	\begin{split}
	V_r^Q &:= \VQr \\
%
%
	\end{split}
\end{equation*}
conditional on
$Z_r^Q$  and $\Delta \w_h^Q$ as
\begin{align}
	\SigmaQCond
	&= \mathrm{E}\big[V_r^Q {V_r^Q}^T | Z_r^Q, \Delta \w_h^Q \big]
	- \mathrm{E}\big[ V_r^Q | Z_r^Q, \Delta \w_h^Q \big] 
	\mathrm{E}\big[V_r^Q | Z_r^Q, \Delta \w_h^Q \big]^T \nonumber \\
%
%
	&= (S_K-I_{K^2}) \Big( \Big( Z_r^Q - \sqrt{\frac{2}{h}} \Delta\w_h^Q \Big)
	\Big( Z_r^Q - \sqrt{\frac{2}{h}} \Delta\w_h^Q \Big)^T \otimes  Q_K \Big) (S_K-I_{K^2})
\end{align}
with $S_K = \sum_{i=1}^K \mathbf{e}_i^T \otimes (I_K \otimes \mathbf{e}_i)$, where
$\mathbf{e}_i$ denotes the $i$-th unity vector. 
This expression can be reformulated without using the operator $S_K$
by taking into account that
\begin{align*}
	&\mathrm{E} \Big[ U_r^Q \Big(Z_r^Q - \sqrt{\frac{2}{h}} \Delta\w_h^Q \Big)^T 
	\otimes \Big(Z_r^Q - \sqrt{\frac{2}{h}} \Delta\w_h^Q \Big) {U_r^Q}^T \Big| Z_r^Q, \Delta \w_h^Q \Big] \\
	&= \Big( I_K \otimes \diag\Big( Z_r^Q - \sqrt{\frac{2}{h}} \Delta\w_h^Q \Big) \Big)
	\big( \mathbf{1}_K^T \otimes (Q_K \otimes \mathbf{1}_K) \big)
	\Big( \diag \Big( Z_r^Q - \sqrt{\frac{2}{h}} \Delta\w_h^Q \Big) \otimes I_K \Big) \\
	&= \Big( Q_K^{1/2} \otimes \diag\Big( Z_r^Q - \sqrt{\frac{2}{h}} \Delta\w_h^Q \Big) \Big)
	\Big( \Big( Z_r^Q - \sqrt{\frac{2}{h}} \Delta\w_h^Q \Big)^T \otimes (Q_K^{1/2} \otimes \mathbf{1}_K) \Big)
\end{align*}
as
\begin{align}
	&\SigmaQCond \nonumber \\
	&= Q_K \otimes \Big( Z_r^Q - \sqrt{\frac{2}{h}} \Delta\w_h^Q \Big) 
	\Big( Z_r^Q - \sqrt{\frac{2}{h}} \Delta\w_h^Q \Big)^T
	+ \Big( Z_r^Q - \sqrt{\frac{2}{h}} \Delta\w_h^Q \Big)
	\Big( Z_r^Q - \sqrt{\frac{2}{h}} \Delta\w_h^Q \Big)^T \otimes Q_K \nonumber \\
	&\quad - \Big( Q_K^{1/2} \otimes \diag\Big( Z_r^Q - \sqrt{\frac{2}{h}} \Delta\w_h^Q \Big) \Big)
	\Big( \Big( Z_r^Q - \sqrt{\frac{2}{h}} \Delta\w_h^Q \Big)^T \otimes (Q_K^{1/2} \otimes \mathbf{1}_K) \Big)
%
	\nonumber \\
	&\quad - \Big( \Big( Z_r^Q - \sqrt{\frac{2}{h}} \Delta\w_h^Q \Big) \otimes (Q_K^{1/2} \otimes \mathbf{1}_K^T) \Big)
	\Big( Q_K^{1/2} \otimes \diag\Big( Z_r^Q - \sqrt{\frac{2}{h}} \Delta\w_h^Q \Big) \Big) 
%
	.
\end{align}
Analogously to \cite{MR1843055},
by taking the expectation, we define
\begin{align} \label{SigmaInf}
	\SQinfty &= \mathrm{E} \Big[ H_K \Sigma^Q(V_1^Q)_{|Z_1^Q, \Delta \w_h^Q} H_K^T 
	\Big| \Delta \w_h^Q \Big] \nonumber \\
	&= 
	2 H_K (Q_K \otimes Q_K) H_K^T
	+ \frac{2}{h} H_K (I_{K^2}-S_K) \big( Q_K \otimes \big( 
	\Delta\w_h^Q {\Delta\w_h^Q}^T \big) \big)
	(I_{K^2}-S_K) H_K^T. 
%
\end{align}
Taking into consideration that
\begin{align*}
	&\mathrm{E} \Big[ \Big( I_K \otimes \diag\Big( Z_r^Q - \sqrt{\frac{2}{h}} \Delta\w_h^Q \Big) \Big)
	\big( \mathbf{1}_K^T \otimes (Q_K \otimes \mathbf{1}_K) \big)
	\Big( \diag \Big( Z_r^Q - \sqrt{\frac{2}{h}} \Delta\w_h^Q \Big) \otimes I_K \Big) \Big| \Delta \w_h^Q \Big] \\
	&= \frac{2}{h} \Big( I_K \otimes \diag\big( \Delta\w_h^Q \big) \Big)
	\big( \mathbf{1}_K^T \otimes (Q_K \otimes \mathbf{1}_K) \big)
	\Big( \diag \big( \Delta\w_h^Q \big) \otimes I_K \Big) \\
	&\quad + \sum_{i=1}^K \big( Q_K^{1/2} \mathrm{e}_i \big)^T \otimes 
	\big( I_K \otimes Q_K^{1/2} \mathrm{e}_i \big)
\end{align*}
and that $H_K \big( \sum_{i=1}^K \big( Q_K^{1/2} \mathrm{e}_i \big)^T \otimes 
\big( I_K \otimes Q_K^{1/2} \mathrm{e}_i \big) \big) H_K^T = 0$, it follows that
expression \eqref{SigmaInf} can be rewritten as
\begin{align} \label{SigmaInf2}
	\SQinfty 
	&=2 H_K (Q_K \otimes Q_K) H_K^T 
	+ \frac{2}{h} H_K \Big( Q_K \otimes \Delta\w_h^Q {\Delta \w_h^Q}^T
	+ \Delta\w_h^Q {\Delta \w_h^Q}^T \otimes Q_K \nonumber \\
	& \quad - \big(Q_K^{{1}/{2}} \otimes \diag(\Delta\w_h^Q) \big)
	\big( {\Delta\w_h^Q}^T \otimes (Q_K^{{1}/{2}} \otimes \mathbf{1}_K) \big) \nonumber \\
	& \quad - \big( \Delta\w_h^Q \otimes (Q_K^{{1}/{2}} \otimes \mathbf{1}_K^T) \big)
	\big( Q_K^{{1}/{2}} \otimes \diag(\Delta\w_h^Q) \big) \Big) H_K^T.
\end{align}
This implies that, given $Z^Q=(Z_r^Q)_{r \in \mathbb{N}}$
and $\Delta \w_h^Q$, the vector of tail sums $\tilde{R}^{Q,(D)}(h)$  
is conditionally Gaussian distributed with the following parameters
\begin{equation*}
 	\tilde{R}^{Q,(D)}(h)_{|Z^Q, \Delta \w_h^Q} 
 	\sim N \Big(0_{L}, 
 	\Big( \frac{h}{2\pi} \Big)^2 \sum_{r=D+1}^{\infty}
 	\frac{1}{r^2} H_K \SigmaQCond H_K^T 
 	\Big)
\end{equation*}
for $D \in \mathbb{N}$.
Hence, given $Z^Q$ and $\Delta \w_h^Q$, we can approximate 
the tail sums by simulating 
a conditionally standard Gaussian random vector
${\Upsilon^{Q,(D)}}_{|Z^Q,\Delta \w_h^Q} \sim N(0_{L}, I_{L})$ defined as
\begin{equation*}
 	\Upsilon^{Q,(D)} = \frac{2\pi}{h}  \bigg( \sum_{r={D+1}}^{\infty} \frac{1}{r^2}
 	H_K \SigmaQCond H_K^T \bigg) ^{-\frac{1}{2}} \tilde{R}^{Q,(D)}(h)
\end{equation*}
and, therewith, obtain the vector of tail sums 
\begin{equation}\label{ApproxRemainder}
 	\tilde{R}^{Q,(D)}(h) = \frac{h}{2\pi} \bigg(\sum_{r={D+1}}^{\infty}  \frac{1}{r^2}
	H_K \SigmaQCond H_K^T \bigg)^{\frac{1}{2}} \Upsilon^{Q,(D)} .
\end{equation}
It remains to examine, how the covariance matrix evolves as $D \to \infty$.
For $D \in \mathbb{N}$, we define the matrix
\begin{equation}\label{SigmaD}
 	\SigmaDQ := \bigg( \sum_{r=D+1}^{\infty} \frac{1}{r^2} \bigg)^{-1} 
 	\sum_{r=D+1}^{\infty} \frac{H_K \SigmaQCond H_K^T}{r^2} .
\end{equation}
%
%
%
By the proof of Theorem~\ref{Algo2} below,
we get convergence in the following sense
\begin{equation*}
	\lim_{D \to \infty} \mathrm{E} \Big[ \big\| \SigmaDQ - \SQinfty \big\|_F^2 \Big] = 0,
\end{equation*}
where $\|\cdot\|_F$ denotes the Frobenius norm. Thus, it follows
\begin{equation*}
 	\frac{2\pi}{h} \bigg( \sum_{r={D+1}}^{\infty} \frac{1}{r^2} \bigg)^{-\frac{1}{2}} \tilde{R}^{Q,(D)}(h)
 	\stackrel{d}{\longrightarrow} \mathbf{\xi} \sim N\big(0_{L}, \SQinfty \big)
\end{equation*}
as $D \to \infty$, see also \cite{MR1843055}. \\ \\
\phantomsection \label{Sec:Algo2}\noindent
Combining the above, we obtain an algorithm very similar to the one in \cite{MR1843055},
where steps $1,2$, and $4$ equal Algorithm~1. Additionally, we approximate the vector of
tail sums in step 3. For some $h>0$, $t,t+h \in [0,T]$, and $D,K\in\mathbb{N}$ the Algorithm~2 
is defined as follows:
\begin{sffamily}
\begin{enumerate}
 \item For $j \in \mathcal{J}_K$, simulate the Fourier coefficients 
 $\Delta \w_{h}^j = \langle W_{t+h}-W_t, \tilde{e}_j \rangle_U$ 
 of the increment $W_{t+h}-W_t$ 
 with
 $\Delta \w_h^Q = \big( \Delta \w^1_h, \ldots, \Delta \w^K_h \big)^T$ as
 \begin{equation*}
 		\Delta\w_h^Q = \sqrt{h} \, Q_K^{{1}/{2}} V
 \end{equation*}
 where $V  \sim N(0_K,I_K)$.
 \item Approximate $\tilde{A}^Q(h)$ as
 \begin{equation*}
   		\tilde{A}^{Q,(D)}(h) = H_K \big(Q_K^{{1}/{2}} \otimes Q_K^{{1}/{2}} \big) H_K^T 
   		\ADV
 \end{equation*}
 where $U_r, Z_r \sim N(0_K, I_K)$ are independent.
 \item Simulate $\Upsilon^{Q,(D)} \sim N(0_{L},I_{L})$ and compute
 \begin{equation}\label{ApproxA}
  		\hat{A}^{Q,(D)}(h) = \tilde{A}^{Q,(D)}(h)
  		+ \frac{h}{2\pi} \bigg( \sum_{r=D+1}^{\infty}\frac{1}{r^2} \bigg)^{\frac{1}{2}}
  		\sqrt{\SQinfty} \Upsilon^{Q,(D)} .
 \end{equation}
 \item Compute the approximation $ \text{vec}((\hat{I}^{Q,(D)}(h))^T)$ of $ \text{vec}((I^Q(h)^T)$ as
 \begin{equation*}
 	\text{vec}((\hat{I}^{Q,(D)}(h))^T) = \frac{\Delta\w_h^Q \otimes \Delta\w_h^Q
 	-\text{vec}(h Q_K)}{2}
 	+ (I_{K^2}-S_K) H_K^T \hat{A}^{Q,(D)}(h) 
 \end{equation*}
 with $S_K = \sum_{i=1}^K \mathbf{e}_i^T \otimes (I_K \otimes \mathbf{e}_i)$. 
\end{enumerate}
\end{sffamily}
Note that the matrix $\sqrt{\SQinfty}$ in step 3 is the Cholesky 
decomposition of $\SQinfty$. This expression, which is specified in the following theorem,
can be obtained in closed
form and does not have to be computed numerically.
\begin{thm}[Cholesky Decomposition]\label{CholeskyDecomp}
	Let $\SQinfty$ be defined as in \eqref{SigmaInf} or \eqref{SigmaInf2} with 
	$\Delta\w_h^Q = \sqrt{h} \, Q_K^{{1}/{2}} V$ and let $\Sinfty$ be
	defined by \eqref{SigmaInf} or \eqref{SigmaInf2} with $Q_K=I_K$. 
	Then, it holds
	\begin{equation*}
 		\sqrt{\SQinfty} 
 		= \tilde{Q}_K \frac{\Sinfty + 2 \sqrt{1+ V^T V} I_{K^2}}
		{\sqrt{2} \big(1+\sqrt{1+ V^T V}\big)}.
	\end{equation*}
\end{thm}
\begin{proof}
 For a proof, we refer to Section~\ref{Sec:Proofs}.
\end{proof}
Now, we analyze the error resulting from Algorithm~2. 
In the following theorem, the first
term is the same as in the error estimate of Algorithm~1, see Theorem~\ref{Algo1}.
Due to the second term, the approximations converge with order $1$ in $D$,
which is twice the order that Algorithm~1 attains. However, 
this expression is dependent on $K$
as well. Below, we state an alternative estimate -- there, the exponent of $K$ is not fixed but
dependent on the eigenvalues $\eta_j$, $j\in\mathcal{J}_K$, see Theorem~\ref{Algo2Alternative}.
The algorithm that is superior can only be 
determined in dependence on the operator $Q$. 
\begin{thm}[Convergence for Algorithm~2]\label{Algo2}
	Assume that $Q$ is a nonnegative and symmetric trace class operator 
	and $(W_t)_{t \geq 0}$ is a $Q$-Wiener process.
	Further, let $\Phi \in L(U,H)_{U_0}$ with $\| \Phi  Q^{-\alpha} \|_{L_{HS}(U_0,H)}<C$,
	$\Psi\in L(H,L(Q^{-\alpha}U,H)_{U_0})$ for some $\alpha\in(0,\infty)$, i.e., 
	assumptions (A1) and (A2) are fulfilled. 
	Then, it holds
	 \begin{align*}
 		&\bigg(\mathrm{E} \bigg[ \Big\| \int_t^{t+h} 
 		\Psi\Big( \int_t^s\Phi  \, \mathrm{d}W_r \Big) \, \mathrm{d}W_s
 		- \sum_{i,j \in \mathcal{J}_K} \hat{I}_{(i,j)}^{Q,(D)}(h)
 		\; \Psi\big(\Phi \tilde{e}_i, \tilde{e}_j \big)
 		\Big\|_H^2 \bigg] \bigg)^{\frac{1}{2}} \\ 
 		&\leq C_Q h \Big(\sup_{j\in\mathcal{J} \setminus
 		\mathcal{J}_K} \eta_j \Big)^{\alpha}
 		+  C_Q\frac{h}{D} \sqrt{K^2(K-1)}
 	\end{align*}
 	for some $C_Q>0$ and all $h>0$, $t,t+h\in[0,T]$,  $D,K\in\mathbb{N}$,
 	and $\mathcal{J}_K \subset \mathcal{J}$ with $|\mathcal{J}_K|=K$.
\end{thm}
\begin{proof}
 For a proof, we refer to Section~\ref{Sec:Proofs}.
\end{proof}
For completeness, we state the following error estimate.
Again, this is the estimate that we employ
when incorporating the approximation of the iterated 
integrals into a numerical scheme; see also the
notes on Corollary~\ref{Algo1Lemma}.
\begin{cor}\label{Algo2Lemma}
	Assume that $Q$ is a nonnegative and symmetric trace class operator 
	and $(W_t)_{t \geq 0}$ is a $Q$-Wiener process.
 	Furthermore, let $\Phi  \in L(U,H)_{U_0}$, $\Psi\in L(H,L(U,H)_{U_0})$, 
 	i.e., conditions (B1) and (B2) 
 	are fulfilled.
 	Then, it holds
 	\begin{align*}
 		&\bigg(\mathrm{E}\bigg[ \Big\| \int_t^{t+h}
 		\Psi \Big( \int_t^s \Phi \, \mathrm{d}W_r^K \Big) \, \mathrm{d}W_s^K
 		- \sum_{i,j \in \mathcal{J}_K} \hat{I}_{(i,j)}^{Q,(D)}(h)
 		\; \Psi\big(\Phi \tilde{e}_i, \tilde{e}_j\big) \Big\|_H^2 \bigg] \bigg)^{\frac{1}{2}}
 		\leq  C_Q\frac{h}{D}\sqrt{K^2(K-1)}
 	\end{align*}
 	for some $C_Q > 0$ and all $h>0$, $t,t+h\in[0,T]$,  $D,K\in\mathbb{N}$,
 	and $\mathcal{J}_K \subset \mathcal{J}$ with $|\mathcal{J}_K|=K$.
\end{cor}
 \begin{proof}
The proof of this corollary is detailed in the proof of Theorem~\ref{Algo2}.
\end{proof}
If we assume $\eta_j \leq Cj^{-\rho_Q}$ for $C>0$, $\rho_Q>1$, and all $j\in\{1,\ldots,K\}$,
we can improve the result in Theorem~\ref{Algo2} in the case $\rho_Q <3$. Precisely,
we obtain an error term
that involves the factor $K^{\frac{\rho_Q}{2}}$. The main difference
is that the alternative proof
works with the entries of the covariance matrices explicitly.
A statement
along the lines of Corollary~\ref{Algo2Lemma} can be obtained analogously.
\begin{thm}[Convergence for Algorithm~2]\label{Algo2Alternative}
	Assume that $Q$ is a nonnegative and symmetric trace class operator 
	and $(W_t)_{t \geq 0}$ is a $Q$-Wiener process.
	Further, let $\Phi \in L(U,H)_{U_0}$ with $\| \Phi  Q^{-\alpha}\|_{L_{HS}(U_0,H)}<C$,
	$\Psi\in L(H,L(Q^{-\alpha}U,H)_{U_0})$ for some $\alpha\in(0,\infty)$, i.e., 
	assumptions (A1) and (A2) are fulfilled. Then, it holds
	\begin{align*}
		&\bigg(\mathrm{E}\bigg[\Big\|\int_t^{t+h} 
		\Psi\Big( \int_t^s\Phi  \, \mathrm{d}W_r\Big) \,\mathrm{d}W_s
		- \sum_{i,j \in \mathcal{J}_K} \hat{I}_{(i,j)}^{Q,(D)}(h)
		\; \Psi\big(\Phi \tilde{e}_i, \tilde{e}_j\big)
		\Big\|_H^2\bigg]\bigg)^{\frac{1}{2}} \\ 
		&\leq C_Qh\Big(\sup_{j\in\mathcal{J} \setminus
		\mathcal{J}_K}\eta_j\Big)^{\alpha}
		+  C_Q\frac{h}{D} \Big(\min_{j\in\mathcal{J}_K}\eta_j\Big)^{-\frac{1}{2}}
	\end{align*}
	for some $C_Q>0$ and all $h>0$, $t,t+h\in[0,T]$,  $D,K\in\mathbb{N}$,
	and $\mathcal{J}_K \subset \mathcal{J}$ with $|\mathcal{J}_K|=K$.
\end{thm}
\begin{proof}
 For a proof, we again refer to Section~\ref{Sec:Proofs}.
\end{proof}
\begin{remark}
 Note that if $(W_t)_{t\geq 0}$ is a cylindrical Wiener process,
 we get the same estimate \eqref{ErrDoubleW} as in the finite dimensional case.
\end{remark}
In general, for $h=\frac{T}{M}$, we obtain convergence of 
this algorithm for $K,M\rightarrow \infty$ if
we choose $D>(\min_{j\in\mathcal{J}_K}\eta_j)^{-\frac{1}{2}}h^{1-\theta}$ or, respectively,
$D>\sqrt{K^2(K-1)}h^{1-\theta}$ 
for some $\theta >0$.
For Algorithm~1, we require $D> h^{2-2\theta}$, instead. 
However, we need a more careful choice of
$D$ to maintain the order of convergence in the mean-square sense in $h$ 
for a given numerical scheme
-- we call this convergence rate $\gamma>0$.
Precisely, we have to choose $D \geq h^{1-2\gamma}$ for Algorithm~1 and 
$D \geq h^{\frac{1}{2}-\gamma}(\min_{j\in\mathcal{J}_K}\eta_j)^{-\frac{1}{2}}$, respectively,
$D \geq h^{\frac{1}{2}-\gamma}\sqrt{K^2(K-1)}$ for Algorithm~2.
\section{Proofs}\label{Sec:Proofs}
\subsection{Convergence for Algorithm~1}
\begin{proof}[Proof of Theorem~\ref{Algo1}] 
We determine the error resulting from the 
approximation of the iterated stochastic integral \eqref{DoubleIntSPDE} by Algorithm~1
which also contains the projection of the $Q$-Wiener process in \eqref{QSeriesK}.
Below, we employ error estimates of the following form several times, see also the proof
in \cite{MR3320928}.
It holds
\begin{equation}\label{ErrorWK}
 \begin{split}
   \mathrm{E}\bigg[\Big\|\int_{t}^{t+h} \Phi  \, \mathrm{d}(W_s-W_s^K) \Big\|_H^2\bigg]
   &= \mathrm{E}\bigg[\Big\| \sum_{j\in\mathcal{J}\setminus{\mathcal{J}_K} }
   \int_{t}^{t+h}\Phi \sqrt{\eta_j} \tilde{e}_j \, \mathrm{d}\beta_s^j\Big\|_H^2\bigg] \\ 
   & = \sum_{j\in\mathcal{J}\setminus{\mathcal{J}_K}}\eta_j\int_{t}^{t+h}
   \mathrm{E}\Big[\big\|\Phi Q^{-\alpha}Q^{\alpha}\tilde{e}_j \big\|_H^2\Big]\,
   \mathrm{d}s \\
    & = \sum_{j\in\mathcal{J}\setminus{\mathcal{J}_K} }\eta_j^{2\alpha+1}\int_{t}^{t+h}
    \mathrm{E}\Big[\big\|\Phi Q^{-\alpha}\tilde{e}_j \big\|_H^2\Big]\, \mathrm{d}s  \\
    &\leq \Big(\sup_{j\in \mathcal{J}\setminus\mathcal{J}_K}\eta_j\Big)^{2\alpha} \int_{t}^{t+h}
    \mathrm{E}\Big[\sum_{j \in \mathcal{J}}\eta_j\big\|\Phi Q^{-\alpha}\tilde{e}_j \big\|_H^2\Big]\,
    \mathrm{d}s \\
    & = \Big(\sup_{j\in \mathcal{J}\setminus\mathcal{J}_K}\eta_j\Big)^{2\alpha}\int_{t}^{t+h}
    \mathrm{E}\Big[\big\|\Phi Q^{-\alpha}\big\|_{L_{HS}(U_0,H)}^2\Big]\, \mathrm{d}s,
    \end{split}
\end{equation}
where we used the expression
\begin{equation*}
  \mathrm{d} (W_s-W_s^K)
  = \sum_{j\in\mathcal{J}\setminus{\mathcal{J}_K} }\sqrt{\eta_j}\tilde{e}_j\, \mathrm{d}\beta_s^j
\end{equation*}
for all $s\in[0,T]$, $K\in\mathbb{N}$ in the first step. 
We fix some arbitrary $h>0$, $t, t+h \in [0,T]$, and $K\in\mathbb{N}$ throughout the proof
and decompose the error into several parts
\begin{equation}\label{SplitDouble}
  \begin{split}
  &\mathrm{E}\bigg[\Big\|\int_t^{t+h} \Psi\Big( \int_t^s  \Phi \,\mathrm{d}W_r\Big) \,\mathrm{d}W_s
  - \sum_{i,j \in \mathcal{J}_K} \bar{I}_{(i,j)}^{Q,(D)}(h)
  \; \Psi\big(\Phi \tilde{e}_i, \tilde{e}_j\big) \Big\|_H^2\bigg]\\
  & \leq C\bigg(\mathrm{E}\bigg[\Big\|\int_t^{t+h} \Psi
  \Big( \int_t^s \Phi  \,\mathrm{d}W_r\Big) \,\mathrm{d}W_s
  -\int_t^{t+h} \Psi 
  \Big( \int_t^s\Phi \,\mathrm{d}W_r^K\Big) \,\mathrm{d}W_s\Big\|_H^2\bigg] \\
  &\quad +\mathrm{E}\bigg[\Big\|\int_t^{t+h} \Psi 
  \Big( \int_t^s \Phi  \,\mathrm{d}W_r^K\Big) \,\mathrm{d}W_s
  -\int_t^{t+h} \Psi 
  \Big( \int_t^s\Phi \,\mathrm{d}W_r^K\Big) \,\mathrm{d}W_s^K\Big\|_H^2\bigg] \\
  &\quad+\mathrm{E}\bigg[\Big\|\int_t^{t+h} \Psi 
  \Big( \int_t^s  \Phi \,\mathrm{d}W_r^K\Big) \,\mathrm{d}W_s^K
  - \sum_{i,j \in \mathcal{J}_K} \bar{I}_{(i,j)}^{Q,(D)}(h)
  \; \Psi \big(\Phi \tilde{e}_i, \tilde{e}_j\big)\Big\|_H^2\bigg]\bigg).
  \end{split}
\end{equation}
For now, we neglect the last term in \eqref{SplitDouble} and estimate the other parts.
By It\^{o}'s isometry, the properties (A1) and (A2) of the operators $\Phi $, $\Psi $,
and estimate \eqref{ErrorWK}, we get 
\begin{align*}
  &\mathrm{E}\bigg[\Big\|\int_t^{t+h} \Psi 
  \Big( \int_t^s \Phi  \,\mathrm{d}W_r\Big) \,\mathrm{d}W_s
  -\int_t^{t+h} \Psi 
  \Big( \int_t^s\Phi \,\mathrm{d}W_r^K\Big) \,\mathrm{d}W_s\Big\|_H^2\bigg] \nonumber\\
  &\quad+\mathrm{E}\bigg[\Big\|\int_t^{t+h} \Psi 
  \Big( \int_t^s \Phi  \,\mathrm{d}W_r^K\Big) \,\mathrm{d}W_s
  -\int_t^{t+h} \Psi 
  \Big( \int_t^s\Phi \,\mathrm{d}W_r^K\Big) \,\mathrm{d}W_s^K\Big\|_H^2\bigg] \nonumber\\
  & \leq \int_t^{t+h} \mathrm{E}\bigg[\Big\|\Psi 
  \Big( \int_t^s \Phi  \,\mathrm{d}\big(W_r-W_r^K\big)\Big)
  \Big\|_{L_{HS}(U_0,H)}^2\bigg]\,\mathrm{d}s \nonumber\\
  &\quad+\Big(\sup_{j\in\mathcal{J}\setminus \mathcal{J}_K} \eta_j\Big)^{2\alpha}
  \int_t^{t+h} \mathrm{E}\bigg[\Big\|\Psi 
  \Big( \int_t^s \Phi  \,\mathrm{d}W_r^K\Big)Q^{-\alpha}
  \Big\|_{L_{HS}(U_0,H)}^2\bigg] \,\mathrm{d}s \nonumber \\
  & \leq C \int_t^{t+h} \mathrm{E}\bigg[\Big\| \int_t^s \Phi  \,\mathrm{d}\big(W_r-W_r^K\big)
  \Big\|_{H}^2\bigg]\,\mathrm{d}s 
  +\Big(\sup_{j\in\mathcal{J}\setminus \mathcal{J}_K} \eta_j\Big)^{2\alpha}
  \int_t^{t+h} \mathrm{E}\bigg[\Big\|\int_t^s \Phi \,\mathrm{d}W_r^K\Big\|_{H}^2
  \bigg] \,\mathrm{d}s \nonumber \\
  & \leq C\Big(\sup_{j\in\mathcal{J}\setminus \mathcal{J}_K}
  \eta_j\Big)^{2\alpha} \int_t^{t+h} \int_t^s
  \mathrm{E}\Big[\big\| \Phi  Q^{-\alpha}\big\|_{L_{HS}(U_0,H)}^2\Big]\,\mathrm{d}r\,\mathrm{d}s 
  +  \Big(\sup_{j\in\mathcal{J}\setminus \mathcal{J}_K} \eta_j\Big)^{2\alpha}
  \int_t^{t+h} \int_t^s C \,\mathrm{d}r \,\mathrm{d}s.\nonumber
\end{align*}
Finally, assumption (A1) yields 
\begin{align}\label{W-WK}
  &\mathrm{E}\bigg[\Big\|\int_t^{t+h} \Psi 
  \Big( \int_t^s \Phi  \,\mathrm{d}W_r\Big) \,\mathrm{d}W_s
  -\int_t^{t+h} \Psi 
  \Big( \int_t^s\Phi \,\mathrm{d}W_r^K\Big) \,\mathrm{d}W_s\Big\|_H^2\bigg] \nonumber\\
  &\quad+\mathrm{E}\bigg[\Big\|\int_t^{t+h} \Psi 
  \Big( \int_t^s \Phi  \,\mathrm{d}W_r^K\Big) \,\mathrm{d}W_s
  -\int_t^{t+h} \Psi 
  \Big( \int_t^s\Phi \,\mathrm{d}W_r^K\Big) \,\mathrm{d}W_s^K\Big\|_H^2\bigg] \nonumber\\
  & \leq C\Big(\sup_{j\in\mathcal{J}\setminus \mathcal{J}_K} \eta_j\Big)^{2\alpha} h^2.
\end{align}
Now, we concentrate on the last term in \eqref{SplitDouble}; this part also proves
Corollary~\ref{Algo1Lemma}.
We get
\begin{align*}
  &\mathrm{E}\bigg[\Big\|\int_t^{t+h} \Psi 
  \Big( \int_t^s \Phi  \,\mathrm{d}W_r^K\Big)\, \mathrm{d}W_s^K
  -\sum_{i,j \in \mathcal{J}_K} \bar{I}_{(i,j)}^{Q,(D)}(h)
                           \; \Psi \big(\Phi \tilde{e}_i, \tilde{e}_j\big) \Big\|_H^2\bigg] \\
  &= \mathrm{E}\bigg[\Big\| \sum_{i,j \in \mathcal{J}_K} I_{(i,j)}^Q(h)
        \; \Psi \big(\Phi \tilde{e}_i, \tilde{e}_j\big)
        -\sum_{i,j \in \mathcal{J}_K} \bar{I}_{(i,j)}^{Q,(D)}(h)
        \; \Psi \big(\Phi \tilde{e}_i, \tilde{e}_j\big) \Big\|_H^2\bigg] \\
  &=  \sum_{i,j \in \mathcal{J}_K} \mathrm{E}\Big[\big(I_{(i,j)}^Q(h)
    -\bar{I}_{(i,j)}^{Q,(D)}(h)\big)^2\Big]\big\|\Psi \big(\Phi \tilde{e}_i, \tilde{e}_j\big)\big\|_H^2
\end{align*}
as $\mathrm{E}\Big[ \big(I_{(i,j)}^Q(h)-\bar{I}_{(i,j)}^{Q,(D)}(h) \big) 
\big(I_{(k,l)}^Q(h)-\bar{I}_{(k,l)}^{Q,(D)}(h) \big) \Big] =0$ for all
$i,j,k,l \in\mathcal{J}_K$ with $(i,j) \neq (k,l)$, $K\in\mathbb{N}$,
see~\cite{MR1214374}. 
By assumptions (B1) and (B2), we obtain
\begin{align*}
  &\mathrm{E}\bigg[\Big\|\int_t^{t+h} \Psi 
  \Big( \int_t^s \Phi  \,\mathrm{d}W_r^K\Big)\, \mathrm{d}W_s^K
  -\sum_{i,j \in \mathcal{J}_K} \bar{I}_{(i,j)}^{Q,(D)}(h)
  \; \Psi \big(\Phi \tilde{e}_i, \tilde{e}_j\big) \Big\|_H^2\bigg] \\
  &\leq \sum_{i,j \in \mathcal{J}_K} \mathrm{E}\Big[\big( I_{(i,j)}^Q(h)
  -\bar{I}_{(i,j)}^{Q,(D)}(h)\big)^2\Big]\big\|\Psi\big\|_{L(H,L(U,H))}^2 \big\|\Phi\big\|_{L(U,H)}^2 \\
  &\leq C \sum_{i,j \in \mathcal{J}_K} \mathrm{E}\Big[\big( I_{(i,j)}^Q(h)
  -\bar{I}_{(i,j)}^{Q,(D)}(h)\big)^2\Big].
\end{align*}
Due to the relations \eqref{AandI1}-\eqref{AandI3}, it is enough to examine $\tilde{A}^Q(h)$
and $\tilde{A}^{Q,(D)}(h)$ which implies
\begin{align}\label{ErrorGleich}
  &\mathrm{E}\bigg[\Big\|\int_t^{t+h} \Psi 
  \Big( \int_t^s \Phi  \,\mathrm{d}W_r^K\Big)\, \mathrm{d}W_s^K
    -\sum_{i,j \in \mathcal{J}_K} \bar{I}_{(i,j)}^{Q,(D)}(h)
      \; \Psi \big(\Phi \tilde{e}_i, \tilde{e}_j\big) \Big\|_H^2\bigg] \nonumber\\
  &\leq 2C \sum_{i=1}^L \mathrm{E}\Big[\big( \tilde{A}_{(i)}^Q(h)
    -\tilde{A}_{(i)}^{Q,(D)}(h)\big)^2\Big].
\end{align}
By \eqref{DoubleLevy}, \eqref{AreaIntKP}, and the properties of $a_r^j$, $b_r^j$ for $r\in\mathbb{N}_0$,
$j \in\mathcal{J}_K$, $K\in\mathbb{N}$, we obtain
\begin{align*}
  &\mathrm{E}\bigg[\Big\|\int_t^{t+h} \Psi 
  \Big( \int_t^s \Phi  \,\mathrm{d}W_r^K\Big)\, \mathrm{d}W_s^K
  -\sum_{i,j \in \mathcal{J}_K} \bar{I}_{(i,j)}^{Q,(D)}(h)
   \; \Psi \big(\Phi \tilde{e}_i, \tilde{e}_j\big) \Big\|_H^2\bigg] \nonumber\\
  & \leq 2C \sum_{\substack{i,j \in \mathcal{J}_K \\ i<j}}
     \mathrm{E}\bigg[\Big(  \pi \sum_{r=D+1}^{\infty} r
       \Big(a_r^i \Big( b^j_r-\frac{1}{\pi r} \Delta\w_h^j \Big) 
       - \Big(b^i_r-\frac{1}{\pi r} \Delta\w_h^i \Big) a^j_r \Big) \Big)^2\bigg] \\
  &= 2C \pi^2 \sum_{\substack{i,j \in \mathcal{J}_K \\ i<j}}
     \sum_{r=D+1}^{\infty} r^2 \, \mathrm{E}\Big[\Big(a_r^i b^j_r - a_r^i \frac{1}{\pi r} \Delta \w_h^j \Big)^2
     + \Big(b^i_r a^j_r - \frac{1}{\pi r} \Delta \w_h^i a^j_r \Big)^2 \Big] \\
  &= 3 C \frac{h^2}{\pi^2} \sum_{\substack{i,j \in \mathcal{J}_K \\ i<j}}
      \eta_i \eta_j \sum_{r=D+1}^{\infty} \frac{1}{r^2}
  \leq 3C \frac{h^2}{\pi^2} \, (\tr Q)^2 \sum_{r=D+1}^{\infty} \frac{1}{r^2}
\end{align*}
for all $D\in\mathbb{N}$.
As in \cite{MR1178485}, we finally estimate
\begin{equation*}
  \sum_{r=D+1}^{\infty}\frac{1}{r^2} \leq \int_D^{\infty}\frac{1}{s^2}\,\mathrm{d}s =\frac{1}{D}
\end{equation*}
and, in total, we obtain for this part
\begin{align}\label{ErrorDoubleK}
  &\mathrm{E}\bigg[\Big\|\int_t^{t+h} \Psi
  \Big( \int_t^s \Phi  \,\mathrm{d}W_r^K\Big)\, \mathrm{d}W_s^K
  -\sum_{i,j \in \mathcal{J}_K} \bar{I}_{(i,j)}^{Q,(D)}(h)
  \; \Psi \big(\Phi \tilde{e}_i, \tilde{e}_j\big) \Big\|_H^2\bigg] 
  \leq 3C \,(\tr Q )^2 \frac{h^2}{D\pi^2}
\end{align}
for all $h>0$, $t, t+h\in[0,T]$, $D,K\in\mathbb{N}$.
\end{proof}
\subsection{Cholesky Decomposition}
\begin{proof}[Proof of Theorem~\ref{CholeskyDecomp}] 
It holds $\SinftyQ = \tilde{Q}_K \Sinfty \tilde{Q}_K^T$, where $\Sinfty$
is given by \eqref{SigmaInf} for $Q_K=I_K$, and
$\Delta \w_h^I = Q_K^{-1/2} \Delta \w_h^Q = \sqrt{h} V$. 
We assume that
\begin{align*}
  \sqrt{\SinftyQ}
  &= \tilde{Q}_K \frac{\Sinfty+2\sqrt{1+V^T V} I_{K^2} }
  {\sqrt{2}\big(1+\sqrt{1+V^T V}\big)}
  = \tilde{Q}_K \frac{\Sinfty+2\sqrt{1+\frac{1}{h} {\Delta \w_h^I}^T \Delta \w_h^I 
  }  I_{K^2} }
  {\sqrt{2}\Big(1+\sqrt{1+\frac{1}{h} {\Delta \w_h^I}^T \Delta \w_h^I
  }\Big)}
\end{align*}
%
%
holds and compute for $a := \sqrt{1+\frac{1}{h} {\Delta \w_h^I}^T \Delta \w_h^I}$
the expression
\begin{align*}
   &\sqrt{\SinftyQ} \sqrt{\SinftyQ}^T
   = \frac{\tilde{Q}_K \Sinfty \big(\Sinfty \big)^T\tilde{Q}_K^T 
   +2a\tilde{Q}_K \Sinfty \tilde{Q}_K^T
   +2a\tilde{Q}_K \big( \Sinfty \big)^T\tilde{Q}_K^T
   +4 a^2\tilde{Q}_K\tilde{Q}_K^T} {2(1+a)^2}\\
   &= \frac{\tilde{Q}_K \Sinfty \big( \Sinfty \big)^T\tilde{Q}_K^T 
   -(2+2a^2)\tilde{Q}_K \Sinfty \tilde{Q}_K^T
   +4 a^2\tilde{Q}_K\tilde{Q}_K^T} {2(1+a)^2}
   + \frac{2+4a+2a^2}{2(1+a)^2} \tilde{Q}_K \Sinfty \tilde{Q}_K^T \\
   &= \frac{\tilde{Q}_K \Sinfty \big( \Sinfty \big)^T\tilde{Q}_K^T 
   -(2+2a^2)\tilde{Q}_K \Sinfty \tilde{Q}_K^T
   +4 a^2\tilde{Q}_K \tilde{Q}_K^T} {2(1+a)^2}
   + \SinftyQ.
\end{align*}
The idea in \cite{MR1843055} is to show that the first term, which slightly differs in \cite{MR1843055},
is zero, i.e.,
\begin{align*}  
  \tilde{Q}_K \big(\Sinfty \big(\Sinfty \big)^T
  -(2+2a^2) \Sinfty +4 a^2 I_L \big)\tilde{Q}_K^T &= 0_{L\times L}
  \\
  \Leftrightarrow \quad \Sinfty\big(\Sinfty\big)^T
  -(2+2a^2)\Sinfty+4 a^2 I_L &= 0_{L\times L},
\end{align*}
which proves that the expression for $\sqrt{\SinftyQ} $ is correct.
In the proof of Theorem 4.1 in \cite{MR1843055}, the author shows
\begin{equation*}
   \Sinfty\big(\Sinfty\big)^T
  -(2+2a^2)\Sinfty +4 a^2 I_L = 0_{L\times L},
\end{equation*}
arguing by the eigenvalues of the minimal polynomial of this equation. We do not repeat this
ideas here but refer to  \cite{MR1843055} for further details. 
\end{proof}
\subsection{Convergence for Algorithm~2}
\begin{proof}[Proof of Theorem~\ref{Algo2}] 
We split the error term as in the proof of Theorem~\ref{Algo1}, 
see equation \eqref{SplitDouble},
and obtain the same expression \eqref{W-WK} 
from the approximation of the $Q$-Wiener process
by $(W^K_{t})_{t\in[0,T]}$, $K\in\mathbb{N}$.
Further, we get as in equation \eqref{ErrorGleich}
\begin{align*}
  &\mathrm{E}\bigg[\Big\|\int_t^{t+h} \Psi
  \Big( \int_t^s \Phi \, \mathrm{d}W_r^K\Big) \,\mathrm{d}W_s^K
  - \sum_{i,j \in \mathcal{J}_K} \hat{I}_{(i,j)}^{Q,(D)}(h)
  \; \Psi \big(\Phi \tilde{e}_i, \tilde{e}_j\big)\Big\|_H^2\bigg] \\
  &\leq 2 C \sum_{i=1}^L\mathrm{E}\Big[\big(\tilde{A}_{(i)}^Q(h)
  -\hat{A}_{(i)}^{Q,(D)}(h)\big)^2\Big]
\end{align*}
for all $h>0$, $t, t+h\in[0,T]$, $K\in\mathbb{N}$.
The following part also proves Corollary~\ref{Algo2Lemma}.
Let $\|\cdot\|_F$ denote the Frobenius norm.
With the expressions for $\tilde{R}^{Q,(D)}(h)$ in \eqref{ApproxRemainder},
with $\Sigma^{Q,(D)} = \tilde{Q}_K \Sigma^{I,(D)} \tilde{Q}_K^T$,
$\SinftyQ = \tilde{Q}_K \Sinfty \tilde{Q}_K^T$
where $\Sigma^{I,(D)}$, $\Sinfty$ are given by \eqref{SigmaD} and \eqref{SigmaInf} for $Q_K=I_K$, respectively,
and the definition of the algorithm \eqref{ApproxA}, we obtain
\begin{align} \label{Eqn-Frobenius}
  &\mathrm{E}\bigg[\Big\|\int_t^{t+h} \Psi 
  \Big( \int_t^s\Phi  \, \mathrm{d}W_r^K\Big) \,\mathrm{d}W_s^K
  - \sum_{i,j \in \mathcal{J}_K} \hat{I}_{(i,j)}^{Q,(D)}(h)
  \; \Psi \big(\Phi \tilde{e}_i, \tilde{e}_j\big)\Big\|_H^2\bigg] \nonumber \\
  & \leq2C \sum_{i=1}^L \mathrm{E}\bigg[\Big( \Big( \tilde{R}^{Q,(D)}(h)
  -\frac{h}{2\pi} \Big(\sum_{r=D+1}^{\infty} \frac{1}{r^2}\Big)^{\frac{1}{2}}
  \sqrt{\SinftyQ}\Upsilon^{Q,(D)}\Big)_{(i)}\Big)^2\bigg]\nonumber\\
  &=2C  \sum_{i=1}^L \mathrm{E}\bigg[\Big(  \Big(\frac{h}{2\pi}
  \Big(\sum_{r=D+1}^{\infty}\frac{1}{r^2}
  H_K \SigmaCondQ H_K^T \Big)^{\frac{1}{2}}\Upsilon^{Q,(D)} \nonumber \\
  &\quad -\frac{h}{2\pi} \Big(\sum_{r=D+1}^{\infty}
  \frac{1}{r^2}\Big)^{\frac{1}{2}} \tilde{Q}_K
  \sqrt{\Sinfty} \Upsilon^{Q,(D)}\Big)_{(i)}\Big)^2\bigg]\nonumber\\
  &= \frac{Ch^2}{2\pi^2} \Big(\sum_{r=D+1}^{\infty}\frac{1}{r^2}\Big)
  \sum_{i=1}^L\nonumber\\
  &\quad \cdot \mathrm{E}\Bigg[\bigg( \bigg(\Big( \, \Big(
  \sum_{r=D+1}^{\infty}\frac{1}{r^2} \Big)^{-\frac{1}{2}} \tilde{Q}_K
  \Big(\sum_{r=D+1}^{\infty}\frac{1}{r^2} H_K \SigmaICond H_K^T\Big)^{\frac{1}{2}}
  -\tilde{Q}_K \sqrt{\Sinfty} \Big) \Upsilon^{Q,(D)}\bigg)_{(i)}\bigg)^2\Bigg]
  \nonumber\\
  &= C \frac{h^2}{2\pi^2} \Big(\sum_{r=D+1}^{\infty}\frac{1}{r^2}\Big)
  \sum_{i=1}^L\mathrm{E}\Big[\Big( \Big( \Big( \tilde{Q}_K \big(\sqrt{\Sigma^{I,(D)}}
  -\sqrt{\Sinfty}\big) \Big)\Upsilon^{Q,(D)}\Big)_{(i)}\Big)^2\Big] \nonumber \\
  &= C \frac{h^2}{2\pi^2}  \Big(\sum_{r=D+1}^{\infty}\frac{1}{r^2}\Big)
  \sum_{i=1}^L  \mathrm{E}\bigg[\mathrm{E}\Big[\Big( \Big( \Big(\tilde{Q}_K \big(\sqrt{\Sigma^{I,(D)}}
  -\sqrt{\Sinfty} \big) \Big)
  \Upsilon^{Q,(D)}\Big)_{(i)}\Big)^2\Big\vert Z^Q,\Delta \w_h^Q\Big]\bigg]\nonumber \\
  &= C \frac{h^2}{2\pi^2}  \Big(\sum_{r=D+1}^{\infty}\frac{1}{r^2}\Big)
  \,\mathrm{E}\Big[\big\| \tilde{Q}_K \big(\sqrt{\Sigma^{I,(D)}}
  -\sqrt{\Sinfty}\big) \big\|_F^2\Big]
\end{align}
for all $h>0$, $t, t+h\in[0,T]$, $D,K\in\mathbb{N}$. Here, we used the fact that 
${\Upsilon^{Q,(D)}}_{|Z^Q,\Delta \w_h^Q}\sim N(0_{L}, I_{L})$ for  $h>0$, $D,L\in\mathbb{N}$
and that $\tilde{Q}_K$ is a diagonal matrix.
Precisely, for $G:= \sqrt{\Sigma^{I,(D)}}-\sqrt{\Sinfty}$
with $G:= (g_{ij})_{1\leq i,j\leq L}$ and 
$\Upsilon^{Q,(D)} = (\Upsilon^{Q,(D)}_j)_{1\leq j\leq L}$, we compute
\begin{align*}
  &\sum_{i=1}^L \mathrm{E} \bigg[ \mathrm{E} \Big[ \Big( \Big( \Big( \tilde{Q}_K \big(\sqrt{\Sigma^{I,(D)}}
  -\sqrt{\Sinfty} \big) \Big)\Upsilon^{Q,(D)} \Big)_{(i)} \Big)^2
  \Big\vert Z^Q,\Delta \w_h^Q\Big]\bigg]\\ 
  &= \sum_{i=1}^L \mathrm{E} \bigg[ \mathrm{E} \Big[ \Big( \big( \tilde{Q}_K G
  \Upsilon^{Q,(D)} \big)_{(i)} \Big)^2
  \Big\vert Z^Q,\Delta \w_h^Q\Big]\bigg] \\
  &= \sum_{i=1}^L \mathrm{E}\bigg[\mathrm{E}\Big[ 
  \Big( \sum_{j=1}^L (\tilde{Q}_K)_{ii} \, g_{ij} \, \Upsilon^{Q,(D)}_j
  \Big)^2 \Big\vert Z^Q,\Delta \w_h^Q\Big]\bigg] \\
  &= \sum_{i,j=1}^L \mathrm{E}\Big[ (\tilde{Q}_K)_{ii}^2 \, g_{ij}^2 \Big] 
  = \mathrm{E}\Big[ \big\| \tilde{Q}_K G \big\|_F^2\Big] \\
  &= \mathrm{E} \Big[ \big\| \tilde{Q}_K \big( \sqrt{\Sigma^{I,(D)}}
  -\sqrt{\Sinfty}\big) \big\|_F^2\Big].
\end{align*}
In order to relate to the proof in \cite{MR1843055}, we write
\begin{align*}
  \mathrm{E} \Big[ \big\| \tilde{Q}_K \big(\sqrt{\Sigma^{I,(D)}}
  -\sqrt{\Sinfty}\big) \big\|_F^2\Big] 
  &=\mathrm{E} \Big[ \sum_{i,j=1}^L (\tilde{Q}_K)_{ii}^2 g_{ij}^2\Big] 
  \leq \max_{1\leq i \leq K} \eta_i^2 \, \mathrm{E}\Big[ \sum_{i,j=1}^L g_{ij}^2\Big]
  \\ 
  &\leq \max_{1\leq i \leq K} \eta_i^2 \, \mathrm{E}\Big[\big\| \sqrt{\Sigma^{I,(D)}}-\sqrt{\Sinfty}\big\|_F^2\Big].
\end{align*}
In total, we obtain
\begin{align*}
  &\mathrm{E}\bigg[\Big\|\int_t^{t+h} \Psi 
  \Big( \int_t^s\Phi  \, \mathrm{d}W_r^K\Big) \,\mathrm{d}W_s^K
  - \sum_{i,j \in \mathcal{J}_K} \hat{I}_{(i,j)}^{Q,(D)}(h)
  \; \Psi \big(\Phi \tilde{e}_i, \tilde{e}_j\big)\Big\|_H^2\bigg] \\
  & \leq C \max_{1\leq i \leq K} \eta_i^2 \, \frac{h^2}{2\pi^2}  \Big(\sum_{r=D+1}^{\infty}\frac{1}{r^2}\Big)
  \,\mathrm{E}\Big[\big\| \sqrt{\Sigma^{I,(D)}}-\sqrt{\Sinfty}\big\|_F^2\Big].
\end{align*}
Now, we can insert the results obtained in the proofs of \cite[Theorem~4.1, Theorem~4.2, Theorem~4.3]{MR1843055};
this yields 
\begin{align*}
  &\mathrm{E}\bigg[\Big\|\int_t^{t+h} \Psi 
  \Big( \int_t^s\Phi  \, \mathrm{d}W_r^K\Big) \,\mathrm{d}W_s^K
  - \sum_{i,j \in \mathcal{J}_K} \hat{I}_{(i,j)}^{Q,(D)}(h)
  \; \Psi \big(\Phi \tilde{e}_i, \tilde{e}_j\big)\Big\|_H^2\bigg] \\
  & \leq C \max_{1\leq i \leq K} \eta_i^2 \,
  \frac{h^2 K(K-1) \big(K+4 
  \mathrm{E}\big[V^T V\big]
  \big)}{12 \pi^2 D^2}
  \leq C \frac{5 h^2 K^2 (K-1)}{12 \pi^2 D^2}
\end{align*}
for all $h>0$, $t, t+h\in[0,T]$, $D,K\in\mathbb{N}$ where $V = h^{-1/2} Q_K^{-1/2} \Delta \w_h^Q$.
\end{proof}
\vspace{0.5cm}
\begin{proof}[Proof of Theorem~\ref{Algo2Alternative}] 
We split the error term as in the proof of Theorem~\ref{Algo1} and Theorem~\ref{Algo2}, 
see equation \eqref{SplitDouble},
and obtain the same expression \eqref{W-WK} 
from the approximation of the $Q$-Wiener process
by $(W^K_{t})_{t\in[0,T]}$, $K\in\mathbb{N}$.
Moreover, as in the previous proof, we get from \eqref{Eqn-Frobenius} that
\begin{align}\label{SqrtSigma}
  &\mathrm{E}\bigg[\Big\|\int_t^{t+h} \Psi 
  \Big( \int_t^s\Phi  \, \mathrm{d}W_r^K\Big) \,\mathrm{d}W_s^K
  - \sum_{i,j \in \mathcal{J}_K} \hat{I}_{(i,j)}^{Q,(D)}(h)
  \; \Psi \big(\Phi \tilde{e}_i, \tilde{e}_j\big)\Big\|_H^2\bigg] \nonumber \\
  & \leq C \frac{h^2}{2\pi^2}  \Big(\sum_{r=D+1}^{\infty}\frac{1}{r^2}\Big)
  \,\mathrm{E}\Big[\big\| \sqrt{\SigmaDQ}-\sqrt{\SinftyQ} \big\|_F^2\Big]
\end{align}
for all $h>0$, $t, t+h\in[0,T]$, $D,K\in\mathbb{N}$. 
In this alternative proof, we consider the elements of the matrices $\SigmaDQ$ and $\SinftyQ$ explicitly.
Therefore, we define the index set of interest as 
$\mathcal{I}_A  = ((1,2),\ldots,(1,K),\ldots,(l,l+1),\ldots,(l,K),\ldots,(K-1,K)) 
= (I_1,\ldots,I_{L})$ which selects the same entries of some matrix as the matrix transformation 
by $H_K$ given in \eqref{SelectionMatrix}.
The $L \times L$-matrix $H_K \Sigma^Q(V_1^Q)_{|Z_1^Q,\Delta\w_h^Q} H_K^T $ 
has entries of type
\begin{align*}
  & \mathrm{E}\Big[ \big( U_{1i}^Q(Z_{1j}^Q
  -\sqrt{\frac{2}{h}}\Delta\w^j_h)
  -(Z_{1i}^Q-\sqrt{\frac{2}{h}}\Delta\w^i_h)U_{1j}^Q \big) \\ 
  &\quad \cdot
  \big( U_{1m}^Q(Z_{1n}^Q-\sqrt{\frac{2}{h}}\Delta\w^n_h)
  -(Z_{1m}^Q-\sqrt{\frac{2}{h}}\Delta\w^m_h)U_{1n}^Q \big)
  \Big|Z_1^Q,\Delta w^Q_h \Big] 
\end{align*}
for some $i,j,m,n \in \{1, \ldots, K\}$ with $i<j$ and $m<n$.
Especially, its diagonal entries are of type
\begin{equation*}
 \eta_{i}\Big(Z_{1j}
 +\sqrt{\frac{2}{h}}\Delta \w^{j}_h\Big)^2
 + \eta_{j}\Big(Z_{1i}+\sqrt{\frac{2}{h}}\Delta \w^{i}_h\Big)^2
\end{equation*}
with $(i,j) \in \mathcal{I}_A$ and $i \neq j$.
%
%
The off-diagonal entries of the matrix $H_K \Sigma^Q(V_1^Q)_{|Z_1^Q,\Delta\w_h^Q} H_K^T$ 
are of the form
\begin{align*}
  &\mathrm{E}\Big[\big(U_{1i}^Q(Z_{1j}^Q
  -\sqrt{\frac{2}{h}}\Delta\w^j_h)
  -(Z_{1i}^Q-\sqrt{\frac{2}{h}}\Delta\w^i_h)U_{1j}^Q\big)\\ & \qquad \cdot
  \big(U_{1m}^Q(Z_{1n}^Q-\sqrt{\frac{2}{h}}\Delta\w^n_h)
  -(Z_{1m}^Q-\sqrt{\frac{2}{h}}\Delta\w^m_h)U_{1n}^Q\big)\Big|Z_1^Q,\Delta w^Q_h\Big]\\
  &= \left\{ \begin{matrix*}[{l}]
            \quad 0  ,     & i,j \notin \{m,n\} \\
            \quad \eta_i \big(Z_{1j}^Q-\sqrt{\frac{2}{h}}\Delta \w_h^j\big)
            \big(Z_{1n}^Q-\sqrt{\frac{2}{h}}\Delta \w_h^n\big), & i = m, \, j \neq n \\
             -\eta_i\big(Z_{1j}^Q-\sqrt{\frac{2}{h}}\Delta \w_h^j\big)
            \big(Z_{1m}^Q-\sqrt{\frac{2}{h}}\Delta \w_h^m\big), & i=n, \, j \neq m \\
             -\eta_j\big(Z_{1i}^Q-\sqrt{\frac{2}{h}}\Delta \w_h^i\big)
            \big(Z_{1n}^Q-\sqrt{\frac{2}{h}}\Delta \w_h^n\big), & j=m, \, i \neq n \\
            \quad \eta_j\big(Z_{1i}^Q-\sqrt{\frac{2}{h}}\Delta \w_h^i\big)
            \big(Z_{1m}^Q-\sqrt{\frac{2}{h}}\Delta \w_h^m\big), & j=n, \, i \neq m 
            \end{matrix*} \right.
\end{align*}
with $i,j,m,n \in \{1,\ldots,K\}$, $i < j$ and $m < n$.
Therewith, it is easy to see that for  
$\SinftyQ = \mathrm{E}\Big[H_K \Sigma^Q(V_1^Q)_{|Z_1^Q,\Delta \w_h^Q} H_K^T 
\Big|\Delta \w_h^Q\Big]$, 
we get
\begin{align*}
  \big(\SinftyQ\big)_{(k,k)} = 2\eta_{i}\eta_{j} +\frac{2}{h}\eta_{i}(\Delta\w_h^{j})^2
  +\frac{2}{h}\eta_{j}(\Delta\w_h^{i})^2
\end{align*}
and for the off-diagonal entries, it holds
\begin{align*}
    \big(\SinftyQ\big)_{(k,l)}  = \left\{ \begin{matrix*}[l]
			   0  ,     & i,j \notin \{m,n\} \\
			  \frac{2}{h}\eta_i
			  \Delta\w_h^j\Delta\w_h^n, & i= m, \, j \neq n \\
			  - \frac{2}{h}\eta_i
			  \Delta\w_h^j\Delta\w_h^m, & i=n, \, j \neq m \\
			  -\frac{2}{h} \eta_j
			  \Delta\w_h^i\Delta\w_h^n, & j=m, \, i \neq n \\
			  \frac{2}{h} \eta_j
			  \Delta\w_h^i\Delta\w_h^m ,& j=n, \, i \neq m 
			  \end{matrix*} \right.
\end{align*}
with 
$k,l\in\{1,\ldots, L\}$, $l \neq k$,
$i,j,m,n\in\{1,\ldots,K\}$, $i < j$, and $m < n$.
Next, we employ the following lemma
from \cite{MR1843055} in order to rewrite \eqref{SqrtSigma}.
\begin{lma} \label{CholEV}
  Let $A$ and $G$ be symmetric positive definite matrices and denote the smallest eigenvalue
  of matrix $G$ by $\lambda_{min}$. Then, it holds
  \begin{equation*}
    \|A^{\frac{1}{2}}-G^{\frac{1}{2}}\|_F^2 \leq \frac{1}{\sqrt{\lambda_{min}}} \|A-G\|_F^2.
  \end{equation*}
\end{lma}
\begin{proof}[Proof of Lemma~\ref{CholEV}]
A proof can be found in \cite[Lemma 4.1]{MR1843055}.
\end{proof}
For simplicity, we assume $\eta_1\geq\eta_2\geq\ldots\geq\eta_K$ for all $K\in\mathbb{N}$. 
We decompose $\SinftyQ$ as
\begin{equation*}
  \SinftyQ= 2\eta_{K-1}\eta_KI_{L}+\widehat{\SinftyQ}
\end{equation*}
to determine its smallest eigenvalue.
The matrix $\widehat{\SinftyQ}$ is defined as follows: For the
diagonal elements, we get values 
\begin{equation*}
  \big(\widehat{\SinftyQ}\big)_{(k,k)} = \big(\SinftyQ\big)_{(k,k)}-2\eta_{K-1}\eta_K
  = 2(\eta_i\eta_j-\eta_{K-1}\eta_K)+\frac{2}{h}\eta_i(\Delta \w_h^j)^2
  +\frac{2}{h}\eta_j(\Delta \w_h^i)^2 \geq 0
\end{equation*}
with $k \in\{1,\ldots,L\}$, $(i,j) \in \mathcal{I}_A$, and $h>0$.
For the off-diagonal elements, we get $\big(\widehat{\SinftyQ}\big)_{(k,l)}
= \big({\SinftyQ}\big)_{(k,l)}$ for all $k,l \in\{1,\ldots,L\}$, $k\neq l$.
As the matrix $\widehat{\SinftyQ}$ is symmetric and positive semi-definite, the smallest
eigenvalue $\lambda_{\min}$ of ${\SinftyQ}$ fulfills
$\lambda_{\min} \geq 2\eta_{K-1} \, \eta_K \geq 2\eta_K^2$. \\ \\
Below, we use the notation $c_D = \sum_{r=D+1}^{\infty}\frac{1}{r^2}$ for legibility. 
The matrices $\SigmaDQ$ and $\SinftyQ$
are symmetric positive definite. By Lemma~\ref{CholEV} and the
definitions of $\SigmaDQ$, $\SinftyQ$ in
\eqref{SigmaD} and \eqref{SigmaInf}, respectively,
we obtain from \eqref{SqrtSigma}
\begin{align*}
  &\mathrm{E}\bigg[\Big\|\int_t^{t+h} \Psi
  \Big(  \int_t^s \Phi \, \mathrm{d}W_r^K\Big) \,\mathrm{d}W_s^K
  - \sum_{i,j \in \mathcal{J}_K} \hat{I}_{(i,j)}^{Q,(D)}(h)
  \; \Psi \big(\Phi \tilde{e}_i, \tilde{e}_j\big)\Big\|_H^2\bigg] \\
  &\leq \frac{C h^2 c_D}{2\sqrt{2}\eta_K\pi^2}  
  \mathrm{E}\Big[\big\| \SigmaDQ-\SinftyQ \big\|_F^2\Big]\\
  &= \frac{C h^2 c_D}{2\sqrt{2}\eta_K\pi^2}
  \mathrm{E}\bigg[\Big\| c_D^{-1} 
  \sum_{r=D+1}^{\infty}\frac{1}{r^2} H_K \SigmaCondQ H_K^T
  -\mathrm{E}\Big[ H_K \Sigma^Q(V_1^Q)_{|Z_1^Q,\Delta \w_h^Q} H_K^T \Big|\Delta \w_h^Q\Big] \Big\|_F^2\bigg]\\
  &= \frac{C h^2 c_D}{2\sqrt{2} \eta_K\pi^2} \mathrm{E}\bigg[\Big\|
  c_D^{-1}
  \Big( \sum_{r=D+1}^{\infty} \tfrac{H_K \SigmaCondQ H_K^T}{r^2}
  - \sum_{r=D+1}^{\infty} \tfrac{\mathrm{E}
  \big[H_K \Sigma^Q(V_1^Q)_{|Z_1^Q,\Delta \w_h^Q} H_K^T \big| \Delta \w_h^Q\big]}{r^2} \Big)\Big\|_F^2\bigg]\\
  &= \frac{C h^2 c_D^{-1}}{2\sqrt{2} \eta_K\pi^2} \sum_{k,l=1}^L
  \mathrm{E}\Bigg[ \mathrm{E}\bigg[ \Big( \sum_{r=D+1}^{\infty}
  \tfrac{H_K \SigmaCondQ H_K^T - \mathrm{E}\big[ H_K \Sigma^Q(V_1^Q)_{|Z_1^Q,\Delta \w_h^Q} H_K^T \big| \Delta \w_h^Q\big]}{r^2}
  \Big)_{(k,l)}^2 \Big| \Delta \w_h^Q \bigg] \Bigg]
\end{align*}
for $h>0$, $t, t+h\in[0,T]$, $D,K \in\mathbb{N}$.
Following ideas from \cite{MR1843055}, we get 
\begin{align*}
  &\mathrm{E}\bigg[\Big\|\int_t^{t+h} \Psi 
  \Big( \int_t^s \Phi  \, \mathrm{d}W_r^K\Big) \,\mathrm{d}W_s^K
  - \sum_{i,j \in \mathcal{J}_K} \hat{I}_{(i,j)}^{Q,(D)}(h)
  \; \Psi \big(\Phi \tilde{e}_i, \tilde{e}_j\big)\Big\|_H^2\bigg] \\
  &\leq C \frac{h^2}{2 \sqrt{2} \eta_K \pi^2}  \Big(\sum_{r=D+1}^{\infty}
  \frac{1}{r^2}\Big)^{-1} \sum_{k,l=1}^L
  \mathrm{E}\Big[\operatorname{Var}\Big(\Big( \sum_{r=D+1}^{\infty}\frac{1}{r^2}
  H_K \SigmaCondQ H_K^T \Big)_{(k,l)} \Big| \Delta \w_h^Q\Big)\Big] \\
  &= C \frac{h^2}{2 \sqrt{2} \eta_K \pi^2} \Big(\sum_{r=D+1}^{\infty} \frac{1}{r^2}\Big)^{-1} \
  \sum_{k,l=1}^L \sum_{r=D+1}^{\infty}\frac{1}{r^4}\mathrm{E}\Big[\operatorname{Var}
  \Big( \Big(H_K \SigmaCondQ H_K^T \Big)_{(k,l)} \Big| \Delta \w_h^Q\Big)\Big] .
\end{align*}
Next, we compute the conditional expectation involved in this estimate.
We insert the expressions detailed above for $H_K \SigmaCondQ H_K^T$,
$r\in\mathbb{N}$, and $\SinftyQ$ 
and split the sum into diagonal entries
and off-diagonal elements of the matrix. 
This yields for $h>0$, $t, t+h\in[0,T]$, $D, L\in\mathbb{N}$
\begin{align}\label{EqProof1}
  &\sum_{k,l=1}^L \sum_{r=D+1}^{\infty}\frac{1}{r^4} \mathrm{E}\Big[
  \operatorname{Var}\Big( \Big(H_K \SigmaCondQ H_K^T \Big)_{(k,l)}\Big|\Delta \w_h^Q\Big)\Big]
  \nonumber \\
  &= \sum_{k=1}^L \sum_{r=D+1}^{\infty}\frac{1}{r^4}
  \mathrm{E}\bigg[ \mathrm{E}\Big[ \Big(H_K \SigmaCondQ H_K^T - \SinftyQ\Big)^2_{(k,k)}
  \Big| \Delta \w_h^Q \Big]\bigg] \nonumber \\
  &\quad +\sum_{\substack{k,l=1 \\ k\neq l}}^L \sum_{r=D+1}^{\infty}\frac{1}{r^4}
  \mathrm{E}\bigg[ \mathrm{E}\Big[ \Big(H_K \SigmaCondQ H_K^T - \SinftyQ \Big)^2_{(k,l)}
  \Big| \Delta \w_h^Q\Big]\bigg] \nonumber \\
  &= \sum_{\substack{i,j \in \mathcal{J}_K \\ i<j}} \sum_{r=D+1}^{\infty}\frac{1}{r^4}
  \bigg(\mathrm{E}\bigg[ \mathrm{E}\Big[\Big(\eta_i\Big((Z_{rj}^Q)^2 - 2Z_{rj}^Q\sqrt{\frac{2}{h}}
  \Delta\w_h^j\Big) \nonumber \\
  &\quad \quad +
  \eta_j\Big((Z_{ri}^Q)^2 - 2Z_{ri}^Q\sqrt{\frac{2}{h}}\Delta\w_h^i\Big)-2\eta_i\eta_j\Big)^2
  \Big|\Delta \w_h^Q\Big]\bigg]\bigg) \nonumber \\
  &\quad +\sum_{\substack{i,j,m,n\in\mathcal{J}_K  \\ i<j; \, m<n}} \sum_{r=D+1}^{\infty}\frac{1}{r^4}
  \bigg(\mathrm{E}\bigg[ \mathrm{E}\Big[\eta_i^2\Big( Z_{rj}^QZ_{rm}^Q - Z_{rj}^Q\sqrt{\frac{2}{h}}
  \Delta\w_h^m - Z_{rm}^Q\sqrt{\frac{2}{h}}\Delta\w_h^j\Big)^2
  \mathds{1}_{i=n}\mathds{1}_{j\neq m} \nonumber \\
  &\quad \quad + \eta_i^2 \Big(Z_{rj}^QZ_{rn}^Q - Z_{rj}^Q \sqrt{\frac{2}{h}}\Delta\w_h^n
  - Z_{rn}^Q\sqrt{\frac{2}{h}}\Delta\w_h^j\Big)^2
  \mathds{1}_{i=m}\mathds{1}_{j\neq n}\nonumber \\
  &\quad \quad + \eta_j^2 \Big(Z_{ri}^QZ_{rm}^Q - Z_{ri}^Q \sqrt{\frac{2}{h}}\Delta\w_h^m
  - Z_{rm}^Q\sqrt{\frac{2}{h}}\Delta\w_h^i\Big)^2
  \mathds{1}_{j=n}\mathds{1}_{i\neq m} \nonumber \\
  &\quad \quad + \eta_j^2\Big( Z_{ri}^QZ_{rn}^Q - Z_{ri}^Q \sqrt{\frac{2}{h}}\Delta\w_h^n
  - Z_{rn}^Q\sqrt{\frac{2}{h}}\Delta\w_h^i\Big)^2
  \mathds{1}_{j=m}\mathds{1}_{i\neq n}
  \Big| \Delta \w_h^Q\Big]\bigg]\bigg).
\end{align}
We compute the terms in \eqref{EqProof1} separately and obtain
\begin{align*}
  &\mathrm{E}\bigg[ \mathrm{E}\Big[\Big(\eta_i\Big((Z_{rj}^Q)^2
  - 2Z_{rj}^Q\sqrt{\frac{2}{h}}
  \Delta\w_h^j\Big)+
  \eta_j\Big((Z_{ri}^Q)^2 - 2Z_{ri}^Q\sqrt{\frac{2}{h}}
  \Delta\w_h^i\Big)-2\eta_i\eta_j\Big)^2
  \Big|\Delta \w_h^Q\Big]\bigg]\\
  &= \mathrm{E}\Big[ 3\eta_i^2\eta_j^2 +2\eta_i^2\eta_j^2
  -4\eta_i^2\eta_j^2 +\frac{8}{h}\eta_i^2\eta_j
  (\Delta\w_h^j)^2+3\eta_i^2\eta_j^2 -4\eta_i^2\eta_j^2 
  +\frac{8}{h}\eta_i\eta_j^2 (\Delta\w_h^i)^2
  +4\eta_i^2\eta_j^2 \Big] \\
  &= 20\eta_i^2\eta_j^2
\end{align*}
and
\begin{align*}
  &\mathrm{E}\bigg[ \mathrm{E}\Big[\eta_i^2\Big( Z_{rj}^QZ_{rm}^Q
  - Z_{rj}^Q \sqrt{\frac{2}{h}} \Delta\w_h^m
  - Z_{rm}^Q \sqrt{\frac{2}{h}}\Delta\w_h^j \Big)^2
  \mathds{1}_{i=n}\mathds{1}_{j\neq m}\Big|
  \Delta \w_h^Q\Big]\bigg]\\
  &= \mathrm{E}\Big[ \eta_i^2\Big(\eta_j\eta_m
  +\frac{2}{h}\eta_j(\Delta\w_h^m)^2+\frac{2}{h}
  \eta_m(\Delta \w_h^j)^2\Big) \mathds{1}_{i=n} \mathds{1}_{j\neq m}\Big]
  = 5\eta_i^2\eta_j\eta_m \mathds{1}_{i=n}\mathds{1}_{j\neq m}
\end{align*}
for all $i,j,m,n \in \mathcal{J}_K$ with $i< j$ and $m<n$. 
For the other terms of this type, we get similar results.
Moreover, we compute bounds for the following expressions 
\begin{align*}
  \sum_{r=D+1}^{\infty} \frac{1}{r^4}
  \leq \int_D^{\infty}\frac{1}{s^4}\,\mathrm{d}s = \frac{1}{3D^3},
  \qquad 
  \sum_{r=D+1}^{\infty}\frac{1}{r^2} 
  \geq \int_{D+1}^{\infty}\frac{1}{s^2}\,\mathrm{d}s = \frac{1}{D+1}
\end{align*}
for all $D\in\mathbb{N}$. A combination of these estimates yields
\begin{align*}
  \Big(\sum_{r=D+1}^{\infty}\frac{1}{r^4}\Big)
  \Big(\sum_{r=D+1}^{\infty}\frac{1}{r^2}\Big)^{-1}
  \leq \frac{D+1}{3D^3}\leq \frac{2}{3D^2}
\end{align*}
for all $D\in\mathbb{N}$. At this point, the main difference to Algorithm~1 
arises -- we obtain a higher order of convergence in $D$.
In total, we get 
\begin{align*}
  &\mathrm{E}\bigg[\Big\|\int_t^{t+h} \Psi 
  \Big(\Phi  \int_t^s \, \mathrm{d}W_r^K\Big) \,\mathrm{d}W_s^K
  - \sum_{i,j \in \mathcal{J}_K} \hat{I}_{(i,j)}^{Q,(D)}(h)
  \; \Psi \big(\Phi \tilde{e}_i, \tilde{e}_j\big)\Big\|_H^2\bigg] \\
  &\leq C \frac{h^2}{2 \sqrt{2} \eta_K\pi^2}  \Big(\sum_{r=D+1}^{\infty}\frac{1}{r^2}\Big)^{-1}
  \sum_{r=D+1}^{\infty}\frac{1}{r^4} \,\Bigg(
  \sum_{\substack{i,j\in\mathcal{J}_K \\ i<j}} 20\eta_i^2\eta_j^2 \nonumber \\
  &\quad +\sum_{\substack{i,j,m,n\in\mathcal{J}_K \\ i<j; \, m<n}}
  5\bigg(\eta_i^2\eta_j\eta_m \mathds{1}_{i=n}\mathds{1}_{j\neq m} +
  \eta_i^2\eta_j\eta_n \mathds{1}_{i=m}\mathds{1}_{j\neq n}
  +\eta_j^2\eta_i\eta_m \mathds{1}_{j=n}\mathds{1}_{i\neq m}
  +\eta_j^2\eta_i\eta_n \mathds{1}_{j=m}\mathds{1}_{i\neq n}\bigg)\Bigg)\\
  &\leq C \frac{h^2}{2 \sqrt{2} \eta_K\pi^2} \frac{2}{3D^2}
  \sum_{\substack{i,j\in\mathcal{J}_K \\ i<j}} \Big( 20\eta_i^2\eta_j^2
  + 10 \eta_i^2 \eta_j \sum_{\substack{m\in\mathcal{J}_K \\ m \neq j}} \eta_m
  + 10 \eta_j^2 \eta_i \sum_{\substack{m\in\mathcal{J}_K \\ m\neq i}}
  \eta_m \Big).
\end{align*}
Finally, this implies for all $h>0$, $t, t+h\in[0,T]$, $D,K\in\mathbb{N}$
\begin{align*}
  &\mathrm{E}\bigg[\Big\|\int_t^{t+h} 
  \Psi \Big(\Phi  \int_t^s \, \mathrm{d}W_r^K\Big) \,\mathrm{d}W_s^K
  - \sum_{i,j \in \mathcal{J}_K} \hat{I}_{(i,j)}^{Q,(D)}(h)
  \; \Psi \big(\Phi \tilde{e}_i, \tilde{e}_j\big)\Big\|_H^2\bigg] \nonumber \\
  &\leq C \frac{h^2}{2 \sqrt{2} \eta_K\pi^2} \frac{2}{3D^2}
  \Big(20\Big(\sup_{j\in\mathcal{J}_K} \eta_j \Big)^2\big(\tr Q\big)^2
  + 20 \Big(\sup_{j\in\mathcal{J}_K} \eta_j\Big)
  \big(\tr Q\big)^3 \Big) 
  \leq C_Q \frac{h^2}{\eta_K D^2},
\end{align*}
that is, more generally,
\begin{align*}
  &\mathrm{E}\bigg[\Big\|\int_t^{t+h} 
  \Psi \Big(\Phi  \int_t^s \, \mathrm{d}W_r^K\Big) \,\mathrm{d}W_s^K
  - \sum_{i,j \in \mathcal{J}_K} \hat{I}_{(i,j)}^{Q,(D)}(h)
  \; \Psi \big(\Phi \tilde{e}_i, \tilde{e}_j\big)\Big\|_H^2\bigg]
  \leq C_Q \frac{h^2}{\big(\min_{j\in\mathcal{J}_K}\eta_j\big) D^2}.
\end{align*} 
The statement of the theorem follows by combing this estimate with \eqref{W-WK}.
\end{proof}
\bibliographystyle{plain} 
\bibliography{literatur}
\end{document}